\documentclass[reqno]{amsart}

\usepackage{amsmath}
\usepackage{amsfonts}
\usepackage{amssymb}
\usepackage{color}
\usepackage{mathrsfs}
\usepackage[colorlinks=true]{hyperref}
\usepackage{cleveref}

\newtheorem{thm}{Theorem}[section]
\newtheorem{lem}[thm]{Lemma}
\newtheorem{cor}[thm]{Corollary}
\newtheorem{prop}[thm]{Proposition}

\newtheorem{rem}[thm]{Remark}

\numberwithin{equation}{section}

\newcommand{\Wqb}{W_{q,\mathcal{B}}}
\newcommand{\R}{\mathbb{R}}

\newcommand{\mS}{\mathbb{S}}

\newcommand{\A}{\mathbb{A}}
\newcommand{\B}{\mathbb{B}}
\newcommand{\T}{\mathbb{T}}

\newcommand{\E}{\mathbb{E}}

\newcommand{\ml}{\mathcal{L}}

\newcommand{\Om}{\Omega}

\newcommand{\ve}{\varepsilon}
\newcommand{\rd}{\mathrm{d}}

\newcommand{\bear}{\begin{eqnarray}} 
\newcommand{\eear}{\end{eqnarray}} 
\newcommand{\bean}{\begin{eqnarray*}} 
\newcommand{\eean}{\end{eqnarray*}} 
\newcommand{\bs}{\begin{split}}
\newcommand{\es}{\end{split}}

\newcommand{\dhr}{\mathrel{\lhook\joinrel\relbar\kern-.8ex\joinrel\lhook\joinrel\rightarrow}}

\begin{document}

%\title[Age-Structured Diffusive Populations]{Some Remarks About the Semigroup Associated to Age-Structured Diffusive Populations}

\title[Linearized Stability in Age-Structured Diffusive Populations]{The Principle of Linearized Stability in Age-Structured Diffusive Populations}

\author{Christoph Walker}
\email{walker@ifam.uni-hannover.de}
\address{Leibniz Universit\"at Hannover\\ Institut f\" ur Angewandte Mathematik \\ Welfengarten 1 \\ D--30167 Hannover\\ Germany}
\author{Josef Zehetbauer}
\email{josef.zehetbauer@gmx.at}
%\address{***}
\date{\today}

\begin{abstract}
The principle of linearized stability is established for age-structured diffusive populations incorporating nonlinear death and birth processes. More precisely, asymptotic exponential stability is shown for equilibria for which the semigroup associated with the linearization at the equiblibrium has a negative growth bound. The result is derived in an abstract framework and applied in concrete situations.
\end{abstract}

%47D06 : One-parameter semigroups and linear evolution equations/ 47A10 : Spectrum, resolvent/  35B35 : Stability in context of PDEs / 35M10 : PDEs of mixed type/ 92D25 : Population dynamics (general)

\maketitle

%%%%%%%%%%%%%%%%%%%%%%%%%%%%%%%%%%%
\section{Introduction}
%%%%%%%%%%%%%%%%%%%%%%%%%%%%%%%%%%%

Let $u=u(t,a,x)\ge 0$ denote the density of an age-structured diffusive population at time $t\ge 0$, age $a\in [0,a_m)$ with maximal age $a_m\in (0,\infty]$, and spatial position $x\in\Om\subset \R^n$. A prototype model \cite{GurtinMacCamy_MB81,Langlais88,WebbSpringer} for the evolution of $u$ reads
\begin{subequations}\label{Eu1a}
\begin{align}
\partial_t u+\partial_a u&=\mathrm{div}_x\big(d(a,x)\nabla_xu\big)-m\big(\bar u(t,x),a,x\big)u\ , && t>0\, , &  a\in (0,a_m)\, ,& & x\in\Om\, ,\label{u1a}\\
u(t,0,x)&=\int_0^{a_m} b\big(\bar u(t,x),a,x\big)u(t,a,x)\,\rd a\, ,& & t>0\, , & & & x\in\Om\, ,\label{u2a}\\
\partial_N u(t,a,x)&=0\ ,& & t>0\, , &  a\in (0,a_m)\, ,& & x\in\partial\Om\, ,\label{u3a}\\
u(0,a,x)&=u_0(a,x)\ ,& & &  a\in (0,a_m)\, , & & x\in\Om\,,\label{u4a}
\end{align}
\end{subequations}
where the death and birth rates $m=m(\bar u ,a,x)\ge 0$  and $b=b(\bar u,a,x)\ge 0$, respectively, are smooth functions  possibly depending on the (weighted) local  overall population
$$
\bar u(t,x)=\int_0^{a_m}\nu(a,x) \, u(t,a,x)\, \rd a
$$
with weight $\nu$.
Spatial dispersal is governed by the diffusion term in \eqref{u1a} with speed $d(a,x)>0$. The initial distribution of the population is $u_0=u_0(a,x)\ge 0$, and $N$ denotes the outward unit normal on $\partial \Om$. 

Questions related to well-posedness and qualitative aspects of linear and nonlinear population models incorporating age and spatial structure  have been addressed by many authors under different assumptions and with different techniques. We mention \cite{DelgadoMolinasuarezJDE08,GurtinMacCamy_MB81,KangRuanJDE21,KangRuanJMB21,Langlais88,RhandiSchnaubelt_DCDS99,Rhandi,ThiemeDCDS,WalkerJEPE,WebbSpringer} (and the references therein) though this list is far from being complete. 

In this research we shall investigate stability properties of equilibrium solutions to problems of the form~\eqref{Eu1a} for which we embed the latter in a more abstract framework. To this end, we set
$$
A(a)w:=\mathrm{div}_x\big(d(a,\cdot)\nabla_xw\big)\, ,\quad w\in E_1\, ,
$$
where e.g. $E_1:=\Wqb^2(\Om)$  consists of all  functions $w$ in the Sobolev space  $W_q^{2}(\Om)$ with $q\in (1,\infty)$ satisfying the Neumann boundary condition $\partial_N w=0$ on $\partial\Om$. For a smooth and positive function~$d$ and fixed $a$, the operator $A(a)$ is then the generator of an analytic semigroup in $E_0:=L_q(\Omega)$ with domain $E_1$ (and if $A(a)$ depends smoothly also on $a$ it induces an evolution operator on $E_0$ with regularity subspace $E_1$).
We shall thus focus in the following on the abstract problem
\begin{subequations}\label{P} 
\begin{align}
\partial_t u+ \partial_au \, &=     A(a)u -m\big(\bar u(t),a\big)u \,, \qquad t>0\, ,\quad a\in (0,a_m)\, ,\label{P1}\\ 
u(t,0)&=\int_0^{a_m}b\big(\bar u(t),a\big)\, u(t,a)\, \rd a\,, \qquad t>0\, ,\label{P2} \\
u(0,a)&=  u_0(a)\,, \qquad a\in (0,a_m)\,,
\end{align}
\end{subequations}
for functions $u=u(t,a):\R^+\times [0,a_m)\rightarrow E_0^+$, where $a_m\in (0,\infty]$ and
$$
A(a): E_1\subset E_0\rightarrow E_0
$$ 
is for each $a\in [0,a_m)$ the generator of an analytic semigroup on some Banach lattice $E_0$ with domain $E_1$. Given such a function $u$ we indicate with a bar its (weighted) integral with respect to~$a$; that is,
$$
\bar u(t)=\int_0^{a_m} \nu( \sigma)\, u(t,\sigma)\, \rd \sigma \in E_0
$$
for a fixed function $\nu$ whenever this integral makes sense.
 We shall be more specific about the assumptions when presenting the main results in~\Cref{Sec2}. It is worth emphasizing though that our approach applies also to other differential operators and other boundary conditions than the ones appearing in~\eqref{Eu1a}. Also note that we will treat the case $a_m<\infty$ and $a_m=\infty$ simultaneously.
 
Let us emphasize that the ``elliptic'' operator $A(a)$ and the age derivative $\partial_a$ --~being supplemented with a nonlocal boundary condition~\eqref{P2}~-- act on different ``variables'' which makes the problem intricate.
It is then natural to consider problem~\eqref{P} as an evolution equation in the Banach space $\E_0:=L_1([0,a_m),E_0)$ (actually, on a subspace thereof to have more flexibility with respect to the nonlinearities).
The well-posedness of problems of the form~\eqref{P} in our setting was addressed e.g. in  \cite{WalkerDCDSA10,WebbSpringer} (see \Cref{T1} below). 

Equilibria --~i.e. time-independent solutions~-- of~\eqref{P}  are  determined from
\begin{subequations}\label{EP} 
	\begin{align}
		\partial_a\phi \, &=     A(a)\phi -m(\bar \phi,a)\phi \,, \quad a\in (0,a_m)\, ,\label{1e}\\ 
		\phi(0)&=\int_0^{a_m}b(\bar \phi ,a)\, \phi(a)\, \rd a\,.\label{2e} 
	\end{align}
\end{subequations} 
Clearly, $\phi\equiv 0$ is always an equilibrium. In  previous research we presented fairly general conditions sufficient for the existence of at least one positive smooth non-trivial equilibrium, e.g. by bifurcation methods~\cite{WalkerSIMA09,WalkerJDE10,WalkerCrelle} or using fixed point methods in conical shell ~\cite{WalkerJDE10,WalkerJEPE}. 
The main purpose of this research now is to establish the principle of linearized stability for an equilibrium $\phi$ of~\eqref{P} in the sense that the linearization of~\eqref{P} at $\phi$  controls the stability of $\phi$. Imposing that the nonlinearities are differentiable at $\phi$, the (formal) linearization of~\eqref{P} at $\phi$ is
\begin{subequations}\label{Plinear} 
	\begin{align}
	\partial_t v+ \partial_av \, &=     A(a)v -m\big(\bar \phi,a\big) v -\partial m\big(\bar \phi,a\big)[\bar v(t)]\phi(a)\,, \qquad t>0\, ,\quad a\in (0,a_m)\, ,\\ 
	v(t,0)&=\int_0^{a_m}b\big(\bar \phi,a\big)\, v(t,a)\, \rd a +\int_0^{a_m}\partial b\big(\bar \phi,a\big)[\bar v(t)]\, \phi(a)\, \rd a \,, \qquad t>0\, , \\
	v(0,a)&=  v_0(a)\,, \qquad a\in (0,a_m)\,,
\end{align} 
\end{subequations}
with $\partial$ indicating Fr\'echet derivatives with respect to $\bar\phi$.
Such linear problems were studied in \cite{WebbSpringer,WalkerIUMJ}. It was shown, in particular, that (under suitable assumptions) the corresponding solution is given by a strongly continuous semigroup $(\T_\phi(t))_{t\ge 0}$ on~$\E_0$ (i.e., $v(t)=\T_\phi(t) v_0$) inheriting the regularizing effect from the parabolic character of the operator~\mbox{$\partial_t-A$}. Under the premise  that this semigroup has an exponential decay we shall show herein that the equilibrium $\phi$ is asymptotically exponentially stable. If this condition is not met in the case of the trivial equilibrium  $\phi=0$, the associated semigroup has  asynchronous exponential growth  \cite{WalkerMOFM,WalkerIUMJ} (see also \cite{KangRuan_MA21} for a slightly different situation with nonlocal diffusion). 

We shall emphasize that our approach to investigate linearized stability is adapted from the case without diffusion \cite{PruessNA83} (see also \cite{WebbBook} for a nice exposition of this approach and \cite{PruessJMB81} for the case of a linear birth rate).  The idea of the proof presented herein follows closely the corresponding proof of \cite{PruessNA83,WebbBook}. There are, however, additional difficulties and technicalities that have to be dealt with when diffusion is taken into account.

%%%%%%%%%%%%%%%%%%%%%%%%%%%%%%%%%%%%%%%%%%%%%%%%%%%%
%%%%%%%%%%%%%%%%%%%%%%%%%%%%%%%%%%%%%%%%%%%%%%%%%%%%
\section{Main Result}\label{Sec2}
%%%%%%%%%%%%%%%%%%%%%%%%%%%%%%%%%%%%%%%%%%%%%%%%%%%%
%%%%%%%%%%%%%%%%%%%%%%%%%%%%%%%%%%%%%%%%%%%%%%%%%%%%

We now list our assumptions in detail, state the required well-posedness result, and then present the main result on the asymptotic exponential stability of equilibria.

%%%%%%%%%%%%%%%%%%%%%%%%%%%%%%%%%%%%%%%%%%%%%%%%%%%%
\subsection*{Preliminaries}
%%%%%%%%%%%%%%%%%%%%%%%%%%%%%%%%%%%%%%%%%%%%%%%%%%%%

Let $J:=[0,a_m]$ if $a_m<\infty$ and $J:=[0,\infty)$ if $a_m=\infty$. We write $\ml(E,F)$ for the normed vector space of bounded linear operators from a Banach space $E$ to a Banach space $F$ and set $\ml(E):=\ml(E,E)$.  In the following, $E_0$ is a real Banach lattice ordered by a closed convex cone~$E_0^+$. We let $E_1$ be a densely and compactly embedded subspace of $E_0$, a situation denoted in the following by 
\begin{equation*}%\label{A0}
E_1\stackrel{d}{\dhr} E_0\,.
\end{equation*}
Fixing for $\theta\in (0,1)$ an admissible interpolation functor $(\cdot,\cdot)_\theta$ (see \cite{LQPP}),  we put $E_\theta:= (E_0,E_1)_\theta$ equipped with the order naturally induced by $E_0^+$.
%and note that 
%\begin{equation*}%\label{dhr}
%E_1\stackrel{d}{\dhr} E_{\theta_1}\stackrel{d}{\dhr} E_{\theta_0} \stackrel{d}{\dhr}   E_0\,,\quad 0\le %\theta_0<\theta_1\le 1\,.
%\end{equation*} 
We suppose that there is $\rho>0$ such that
\begin{subequations}\label{A}
\begin{equation}\label{A1a}
A \in  C^\rho\big(J,\mathcal{H}(E_1,E_0)\big)
\end{equation}
and
\begin{equation}\label{A1aa}
	A(a) \ \text{ is resolvent positive for each $a\in J$}\,,
\end{equation}
where $\mathcal{H}(E_1,E_0)$ denotes the subspace of $\ml(E_1,E_0)$ consisting of generators of analytic semigroups on $E_0$ with domain $E_1$.  
Then, due to \eqref{A1a}, $A$ generates a positive parabolic evolution operator 
$$
\{\Pi(a,\sigma)\in\ml(E_0)\,;\, a\in J\,,\, 0\le\sigma\le a\}
$$ 
on $E_0$ with regularity subspace $E_1$ in the sense of \cite[p.45]{LQPP}, see  \cite[II.Corollary~4.4.2]{LQPP} and \cite[II.Thereom~6.4.2]{LQPP}.
The evolution operator satisfies useful stability estimates on the interpolation spaces. We fix $\alpha\in [0,1)$ and assume that there are $M_\alpha\ge 1$ and $\varpi\in\R$ such that
\begin{equation}\label{EO}
	\|\Pi(a,\sigma)\|_{\ml(E_\alpha)}+(a-\sigma)^\alpha\,\|\Pi(a,\sigma)\|_{\ml(E_0,E_\alpha)}\le M_\alpha e^{\varpi (a-\sigma)}\,,\qquad a\in J\,,\quad 0\le \sigma\le a \,,
\end{equation}
(this is automatically satisfied if $a_m<\infty$, see~\cite[II.Lemma 5.1.3]{LQPP}) and
\begin{equation}\label{A4}
	\text{if $a_m=\infty$, then $\varpi<0$}\,.
\end{equation}
We further assume for the birth rate that
\begin{equation}\label{A1b}
\big[\bar v\to b(\bar v,\cdot)]\in C_b^{1-}\big( E_\alpha, L_{\infty}^+\big(J,\ml(E_\alpha,E_0)\big)\big)
\,,
\end{equation}
and for the death rate that
\begin{equation}\label{A1c}
	\big[\bar v\to m(\bar v,\cdot)]\in C_b^{1-}\big( E_\alpha, L_{\infty}^+\big(J,\ml(E_\alpha,E_0)\big)\big)\,,
\end{equation}
where $C_b^{1-}$ stands for locally Lipschitz continuous maps that are bounded on bounded sets. 
Finally, we fix a weight function $\nu$ such that there is $\vartheta\in (0,1)$ with\footnote{If $\alpha\in (0,1)$, it suffices to take $\vartheta=\alpha$.}
\begin{equation}\label{A1d}
	\nu\in L_{1}^+\big(J,\ml(E_\theta)\big)\cap L_{\infty}\big(J,\ml(E_\theta)\big)\,,\quad \theta\in\{0,\alpha,\vartheta\}\,.
\end{equation}
\end{subequations}
Setting $\E_\theta:=L_1(J,E_\theta)$ we shall use in the following the notation
$$
\bar v:=\int_0^{a_m} \nu(a)\,v(a)\,\rd a\in E_\theta\,,\quad v\in \E_\theta\,.
$$
Observe that the properties of the evolution operator $\Pi$ imply for $v_0\in E_0$ and \mbox{$f\in \E_0=L_1(J,E_0)$} that the function $v\in C(J,E_0)$, given by
\begin{equation}\label{VdK}
	v(a)=\Pi(a,0)v_0+\int_0^a\Pi(a,\sigma)\,f(\sigma)\,\rd a\,,\quad a\in J\,,
\end{equation} 
is the {\it mild solution} to the Cauchy problem
$$
\partial_a v=A(a)v+ f(a)\,,\quad a\in \dot{J}:=J\setminus\{0\}\,,\qquad v(0)= v_0\,,
$$ 
and
\begin{equation}\label{evol}
	\Pi(a,s)=\Pi(a,\sigma)\Pi(\sigma,s) \,,\quad 0\le s\le \sigma\le a \in J\,.
\end{equation} 
It worth noting that the assumptions we impose on $A$, $b$, and $m$ are natural and easily checked in concrete applications such as problem~\eqref{Eu1a} (see \Cref{Sec:Exmp} below).
We shall consider~\eqref{P} as an evolution equation in the phase space
$$
\E_\alpha=L_1(J,E_\alpha)
$$
i.e. we consider functions $u:\R^+\rightarrow \E_\alpha$. In the following, given a function $v:\R^+\rightarrow \E_0$  we use interchangeably the notation $v(t)(a)=v(t,a)$ for $t\in \R^+$ and $a\in J$ for convenience.
Let us finally note that~\eqref{A1c} and~\eqref{A1d} imply, in particular, that 
\begin{equation}\label{2}
	F:=\big[ v\mapsto m(\bar v,\cdot)v]\in C_b^{1-}\big(\E_\alpha, \E_0\big)
\end{equation}
and
\begin{equation}\label{3}
	\big[ v\mapsto b(\bar v,\cdot)v]\in C_b^{1-}\big(\E_\alpha, \E_0\big)\,.
\end{equation}

%%%%%%%%%%%%%%%%%%%%%%%%%%%%%%%%%%%%%%%%%%%%%%%%%%%%
\subsection*{Well-Posedness}
%%%%%%%%%%%%%%%%%%%%%%%%%%%%%%%%%%%%%%%%%%%%%%%%%%%%

Questions related to well-posedness of nonlinear problems of the form~\eqref{P} (and even more general equations) were addressed e.g. in~\cite{WalkerDCDSA10}.
Integrating \eqref{P} formally along characteristics  yields that a solution $u:\R^+\rightarrow\E_\alpha$ 
to \eqref{P} with initial value $u_0\in\E_\alpha$ satisfies the fixed point equation
\begin{subequations}\label{100}
  \begin{equation}\label{u}
     u(t,a)\, =\, \left\{ \begin{aligned}
    &\Pi(a,a-t)\, u_0(a-t) + G_{F(u)}(t,a)\, ,& &   a\in J\,,\ 0\le t\le a\, ,\\
    & \Pi(a,0)\, B_u(t-a)+ G_{F(u)}(t,a)\, ,& &  a\in J\, ,\ t>a\, ,
    \end{aligned}
   \right.
    \end{equation}
where 
\begin{equation}\label{400}
G_{v}(t,a):=\int_{(t-a)_+}^{t}\Pi(a,a-t+s)\, v(s,a-t+s)\,\rd s
 \end{equation}
for $v:\R^+\rightarrow \E_0$,
and where $B_u:=u(\cdot,0)$ satisfies the nonlinear Volterra equation 
    \begin{equation}\label{500}
    \begin{split}
    B_u(t)\, & =\, \int_0^t  b(\bar u(t),a)\, \Pi(a,0)\, B_u(t-a)\, \rd
    a\, +\, \int_t^{a_m} b(\bar u(t),a)\, \Pi(a,a-t)\, u_0(a-t)\, \rd a\\
    &\quad +\int_0^{a_m} b(\bar u(t),a) G_{F(u)}(t,a)\,\rd a
	\end{split}
    \end{equation}		
for $t\ge 0$. Here and in the following we put $b(\bar v,a):=0$ whenever $a\notin J$. Note that $u(t,0)=B_u(t)$ for $t\ge 0$ by ~\eqref{u}, while~\eqref{500} ensures 
\begin{equation}\label{500b}
	\begin{split}
		B_u(t)\, & =\, \int_0^{a_m}  b(\bar u(t),a)\, u(t,a)\, \rd
		a\,,\quad t\ge 0\,.
	\end{split}
\end{equation}
\end{subequations}
This is in accordance with the age boundary condition~\eqref{P2}. 

Motivated by these observations we mean in the following by a (local) solution to problem~\eqref{P} a function $u\in C(I,\E_\alpha)$ satisfying~\eqref{100} for $t\in I$, where $I$ is an interval in $\R^+$ containing $0$. We first state a well-posedness result which is mainly due to~\cite{WalkerDCDSA10}.

\begin{prop}\label{T1}
Let $\alpha\in [0,1)$ and suppose \eqref{A}. 
For every $u_0\in \E_\alpha$ there exists a unique maximal solution 
$u=u(\cdot;u_0)\in C\big(I(u_0),\E_\alpha\big)$
to problem~\eqref{P} on some maximal interval of existence $I(u_0)=[0,t^+(u_0))$; that is, $u(\cdot;u_0)$ satisfies~\eqref{100}. 
%Moreover,
%$$
%\bar u\in C^1\big(I(u_0)\setminus\{0\},E_0\big) \cap C\big(I(u_0)\setminus\{0\},E_1\big)\cap %C^{\alpha-\beta}\big(I(u_0),E_\beta\big)
%$$
%for $\beta\in [0,\alpha]$.
 If 
$$
\sup_{t\in I(u_0)\cap [0,T]} \| u(t;u_0)\|_{\E_\alpha} <\infty\,,\quad T>0\,,
$$
 then the solution exists globally, i.e., $I(u_0)=\R^+$.
Finally, if $u_0\in\E_\alpha^+$, then  $u(t;u_0)\in\E_\alpha^+$ for~$t\in I(u_0)$.
\end{prop}

The solution provided by \Cref{T1} features further properties, in particular with respect to regularity. We shall not go into details here and refer to~\cite{WalkerDCDSA10}. We will briefly address the proof of \Cref{T1} in \Cref{Sec:Ex}.

%%%%%%%%%%%%%%%%%%%%%%%%%%%%%%%%%%%%%%%%%%%%%%%%%%%%
\subsection*{Stability of Equilibria}%\label{Sec2}
%%%%%%%%%%%%%%%%%%%%%%%%%%%%%%%%%%%%%%%%%%%%%%%%%%%%

Let $\phi\in\E_1 \cap C(J,E_0)$ be an equilibrium solution to~\eqref{P} (in the sense that it is a mild solution to~\eqref{EP}).
As mentioned before, existence of non-trivial positive smooth equilibria to problems of the form~\eqref{P} was established in previous works~\cite{WalkerSIMA09,WalkerJDE10,WalkerCrelle,WalkerJEPE} to which we refer.

The main purpose of this research is to establish the principle of linearized stability for such an equilibrium. That is, we want to derive information on the stability of $\phi$ from (spectral properties of) the linearized problem~\eqref{Plinear}.
To make things more precise, we now assume further that
\begin{subequations}\label{B}
\begin{equation}\label{B1}
	\begin{split}
	& E_\alpha\rightarrow L_{\infty}\big(J,\ml(E_\alpha,E_0)\big)\,,\ v\mapsto m(v,\cdot) \text{ is differentiable at $\bar\phi\in E_1$}
	\end{split}
\end{equation}
and
\begin{equation}\label{B2}
	\begin{split}
		& E_\alpha\rightarrow L_{\infty}\big(J,\ml(E_\alpha,E_0)\big)\,,\ v\mapsto b(v,\cdot) \text{ is differentiable at $\bar\phi\in E_1$}\,.
	\end{split}
\end{equation}
Moreover, for technical reasons we also assume that (for some $\vartheta\in (0,1)$, see~\eqref{A1d})
\begin{equation}\label{B3}
	\begin{split}
		b(\bar \phi,\cdot)\in   L_{1}\big(J,\ml(E_\theta)\big)\cap L_{\infty}\big(J,\ml(E_\theta)\big)\,,\quad \theta\in\{0,\alpha,\vartheta\}\,,
	\end{split}
\end{equation}
and 
\begin{equation}\label{B2b}
\big[v\mapsto \partial b(\bar \phi,\cdot)[ v]\phi\big]\in \ml\big(E_\theta, \E_\theta\big)\,,\quad \theta\in\{0,\alpha,\vartheta\}\,.
\end{equation}
\end{subequations} 
%and we then use the notation
%$$
%\big[a\mapsto \partial b(\bar \phi,a)[\bar w]\big]\in  L_{\infty}\big(J,\ml(E_\alpha,E_0)\big)
%$$
%for simplicity. Similarly, we use this notation for $m$. 
%and
%\begin{equation}\label{B4}
%	\begin{split}
%		\big[a\mapsto \partial b(\bar \phi,a)[\cdot]\phi(a)\big]\in  L_{1}\big(J,\ml(E_\theta)\big)\cap %L_{\infty}\big(J,\ml(E_\theta)\big)\,,\quad \theta\in [0,1]\,.
%	\end{split}
%\end{equation}
Setting then
$$
(\partial F(\phi) v)(a):= -m\big(\bar \phi,a\big) v(a) -\partial m\big(\bar \phi,a\big)[\bar v]\phi(a)\, ,\qquad a\in (0,a_m)\,,\quad v\in \E_\alpha\,, 
$$
and
\begin{equation}\label{17C}
\mathcal{M}_\phi(v):=\int_0^{a_m}  b(\bar \phi,a)\,v(a)\,\rd a +\int_0^{a_m}  \partial b(\bar \phi,a)[\bar v]\phi(a)\,\rd a\,,\quad v\in\E_\alpha\,,
\end{equation}
we have
\begin{equation}\label{17Cx}
\partial F(\phi)\in\ml(\E_\alpha,\E_0)\,,\qquad \mathcal{M}\in \ml(\E_\theta,E_\theta)\,,\quad \theta\in\{0,\alpha,\vartheta\}\,.
\end{equation}
It then follows from \cite[Theorem~2.8]{WalkerIUMJ}  that, for  $v_0\in\E_0$, the linearized Cauchy problem (see~\eqref{Plinear})
	\begin{align*}
	\partial_t v+ \partial_av \, &=     A(a)v +\partial F(\phi) v\,, \qquad t>0\, ,\quad a\in (0,a_m)\, ,\\ 
	v(t,0)&=\mathcal{M}_\phi(v(t)) \,, \qquad t>0\, , \\
	v(0,a)&=  v_0(a)\,, \qquad a\in (0,a_m)\,.
	\end{align*} 
defines a strongly continuous semigroup $(\T_\phi(t))_{t\ge 0}$ on $\E_0$; that is, $v(t)=\T_\phi(t)v_0$ is its unique (mild) solution in $\E_0$. Moreover, the semigroup  inherits the regularizing properties from the parabolic part in the sense that there are $N_\alpha(\phi)\ge 1$ and $\omega_\alpha(\phi)\in\R$ with
\begin{equation}\label{C}
\|\T_\phi(t)\|_{\ml(\E_\alpha)} +t^{\alpha}\|\T_\phi(t)\|_{\ml(\E_0,\E_\alpha)}\le N_\alpha(\phi)\, e^{-\omega_\alpha(\phi) t} \,,\quad t\ge 0\,.
\end{equation}
We shall give more details on all these facts later in the subsequent sections (in particular, see Proposition~\ref{LemmaH} below). The main result of this work regarding the stability of equilibria now states that the equilibrium $\phi$ is asymptotically exponentially stable in~$\E_\alpha$ provided that $\omega_\alpha(\phi)>0$:

\begin{thm}\label{T2}
	Let $\alpha\in [0,1)$ and suppose \eqref{A}. Let $\phi\in\E_1\cap C(J,E_0)$ be an equilibrium solution to~\eqref{P} such that \eqref{B} is satisfied. Moreover, suppose that
	$\omega_\alpha(\phi)>0$
	in \eqref{C}. Then,  given any $\omega\in (0,\omega_\alpha(\phi))$, there are $r>0$ and $M\ge 1$ such that, for every  $u_0\in \mathbb{B}_{\E_\alpha}(\phi,r)$, the solution $u(\cdot;u_0)$ to~\eqref{P} exists globally and
	$$
	\|u(t;u_0)-\phi\|_{\E_\alpha}\le M\, e^{-\omega t}\,\|u_0-\phi\|_{\E_\alpha} \,,\quad t\ge 0\,.
	$$
	In particular, the equilibrium $\phi$ is asymptotically exponentially stable in~$\E_\alpha$.
\end{thm}

Note that the assumption $\omega_\alpha(\phi)>0$ is equivalent to the assumption that the growth bound of the semigroup $(\T_\phi(t))_{t\ge 0}$ on $\E_\alpha$ is negative. In some cases (e.g. for the trivial equilibrium $\phi=0$ or if $m=m(a)$ is independent of the population) one can show that the semigroup $(\T_\phi(t))_{t\ge 0}$ on $\E_\alpha$ is eventually compact, hence its growth bound coincides with the spectral bound of its generator (see \Cref{R14} and~\Cref{Sec:Exmp}).
 
It is also worth pointing out that~\eqref{B3} and~\eqref{B2b} are not too restrictive with regard to applications since $\phi\in\E_1$. Finally, we emphasize that one can choose $\alpha\in (0,1)$ positive so that the nonlinearities $m$ and $b$ are defined on a smaller space $\E_\alpha$ than $\E_0$. This is due to the fact that we exploit the regularizing effects  induced from the analytic generator~$A$ in~\eqref{P}. \\

The outline of the remainder is as follows: In Section~\ref{Sec:Ex} we briefly sketch how to prove~ \Cref{T1}. We then prepare in Section~\ref{Sec:DeriveLin} the proof of \Cref{T2}. Fixing an equilibrium $\phi$ we derive a formula for the difference $w:=u(\cdot;u^0)-\phi$ based on the linearization~\eqref{Plinear} of problem~\eqref{P}. The main statement in this context is \Cref{intermediate}. In Section~\ref{Sec:Linear} we set the stage for estimating the $\E_\alpha$-norm of $w=u(\cdot;u^0)-\phi$ by focusing on the linearized problem. In particular, we provide properties and a priori estimates of the underlying linear semigroup $(\T_\phi(t))_{t\ge 0}$ associated with~\eqref{Plinear}. The key in this regard is \Cref{LemmaH} which allows us to give in the subsequent Section~\ref{Sec:6} an alternative representation of the difference $w=u(\cdot;u^0)-\phi$ in terms of the semigroup $(\T_\phi(t))_{t\ge 0}$. The previously established a priori estimates then imply the asymptotic stability of the equilibrium~$\phi$ provided the semigroup $(\T_\phi(t))_{t\ge 0}$ has an exponential decay. This yields \Cref{T2}.

Finally, in \Cref{Sec:Exmp} we revisit the concrete problem~\eqref{Eu1a} and present examples to which our results apply.

%%%%%%%%%%%%%%%%%%%%%%%%%%%%%%%%%%%%%%%%%%%%%%%%%%%%%%%%%%%%%%%
%%%%%%%%%%%%%%%%%%%%%%%%%%%%%%%%%%%%%%%%%%%%%%%%%%%%%%%%%%%%%%%

\section{Well-Posedness: Proof of \Cref{T1}}\label{Sec:Ex}

%%%%%%%%%%%%%%%%%%%%%%%%%%%%%%%%%%%%%%%%%%%%%%%%%%%%%%%%%%%%%%%
%%%%%%%%%%%%%%%%%%%%%%%%%%%%%%%%%%%%%%%%%%%%%%%%%%%%%%%%%%%%%%%

\Cref{T1} is a special case of the results shown in~\cite{WalkerDCDSA10} except for the term $G_{F(u)}$ defined in  \eqref{2} and \eqref{400}. However, noticing from~\eqref{EO} that, for $v\in C\big([0,T],\E_0\big)$ and $t\in [0,T]$,
	\begin{equation*}
		\begin{split}
			\|G_{v}(t,\cdot)\|_{\E_\alpha}&\le \int_0^{a_m}\int_{(t-a)_+}^{t}\|\Pi(a,a-t+s)\|_{\ml(E_0,E_\alpha)}\,\| v(s,a-t+s)\|_{E_0}\,\rd s \,\rd a\\
			& \le c(T)\int_0^{a_m}\int_{(t-a)_+}^{t} (t-s)^{-\alpha}\,\| v(s,a-t+s)\|_{E_0}\,\rd s \,\rd a\\
			& \le c(T)\int_0^{t}  (t-s)^{-\alpha}\,\| v(s)\|_{\E_0}\,\rd s \,,
		\end{split}
	\end{equation*}
we infer from the continuity properties of the evolution operator~$\Pi$ (see \cite[p.45]{LQPP}) that
\begin{equation}\label{o9}
[v\mapsto G_{v}]\in \ml\big(C\big([0,T],\E_0\big),C\big([0,T],\E_\alpha\big)\big)\,, 
\end{equation}
hence \eqref{2} implies
$$
	[v\mapsto G_{F(v)}]\in C_b^{1-}\big(C\big([0,T],\E_\alpha\big),C\big([0,T],\E_\alpha\big)\big)\,. 
$$
The well-posedness stated in \Cref{T1} then follows from this and \eqref{3} exactly along the lines of~\cite[Theorem~2.2]{WalkerDCDSA10} by means of  Banach's fixed point theorem. The positivity is shown as in~\cite[Propositon~2]{WalkerDCDSA10}.

%This proves \Cref{T1}.

%%%%%%%%%%%%%%%%%%%%%%%%%%%%%%%%%%%%%%%%%%%%%%%%%%%%%%%%%%%%%%%
%%%%%%%%%%%%%%%%%%%%%%%%%%%%%%%%%%%%%%%%%%%%%%%%%%%%%%%%%%%%%%%

\section{Derivation of the Linearization}\label{Sec:DeriveLin}
	
%%%%%%%%%%%%%%%%%%%%%%%%%%%%%%%%%%%%%%%%%%%%%%%%%%%%%%%%%%%%%%%
%%%%%%%%%%%%%%%%%%%%%%%%%%%%%%%%%%%%%%%%%%%%%%%%%%%%%%%%%%%%%%%

We prepare the proof of	\Cref{T2} by deriving the linearization of problem~\eqref{P} at an equiblibrium. \\

For the remainder of this paper, suppose \eqref{A} and let $\phi\in\E_1 \cap C(J,E_0)$ be a fixed equilibrium solution to~\eqref{P} --~i.e. $\phi$ is a mild solution to~\eqref{EP}~-- such that \eqref{B} is satisfied. 
We first note the following representation of~$\phi$.

\begin{lem}\label{phi}
	The equilibrium $\phi\in\E_1 \cap C(J,E_0)$ satisfies the identity
	\begin{equation}\label{13A}
		\phi(a)\, =\, \left\{ \begin{aligned}
			&\Pi(a,a-t)\, \phi(a-t) + G_{F(\phi)}(t,a)\, ,& &   a\in J\,,\  t\le a\, ,\\
			& \Pi(a,0)\,\phi(0)+ G_{F(\phi)}(t,a)\, ,& &  a\in J\, ,\ t>a\, ,
		\end{aligned}
		\right.
	\end{equation}
 for every $t\ge 0$, where $F$ and $G_{F(\phi)}$ are defined in~\eqref{2} and~\eqref{400}, respectively.
\end{lem}

\begin{proof}
		It  readily follows from \eqref{VdK} and \eqref{1e} that
	\begin{equation}\label{10}
		\phi(a)=\Pi(a,0)\phi(0)+\int_0^a\Pi(a,\sigma)F(\phi)(\sigma)\,\rd\sigma\,,\quad a\in J\,.
	\end{equation}
%and, when this is plugged into~\eqref{2e}, $\phi(0)$ satisfies
%	$$
%	\phi(0)=\int_0^ab(\bar\phi,a)\,\Pi(a,0)\,\rd a\, %\phi(0)+\int_0^ab(\bar\phi,a)\,\int_0^a\Pi(a,\sigma)F(\phi)(\sigma)\,\rd\sigma\rd a\,.
%	$$
	Let $t\ge 0$ be arbitrarily fixed and $a\in J$. If $ a \in (0,t)$, then, by formula \eqref{10},
	\begin{align*}
		\phi (a)	&=  \Pi (a,0) \phi(0) + \int_{t-a}^t \Pi(a, a-t+s) F(\phi) (a-t+s) \, \rd s \\
		&=\Pi (a,0) \phi(0)+G_{F(\phi)}(t,a)\,.
	\end{align*}
	If $t< a_m$ and  $a \in (t, a_m)$, then, by formula \eqref{10} and the evolution property~\eqref{evol},
	\begin{align*}
		\phi (a)	&=  \Pi (a,0) \phi(0) + \left(\int_0^{a-t}+\int_{a-t}^a\right) \Pi(a, \sigma) F(\phi) (\sigma) \, \rd \sigma \\
		&= \Pi (a,a-t)\left(\Pi (a-t,0) \phi(0) + \int_0^{a-t} \Pi (a-t,\sigma) F(\phi) (\sigma) \, \rd\sigma\right) \\
		&\quad + \int_{0}^t \Pi (a,a-t+s) F(\phi) (a-t+s) \, \rd s\\
		&=\Pi (a,a-t) \phi(a-t)+G_{F(\phi)}(t,a)\,.
	\end{align*}
	This is the assertion.
\end{proof}

Let now $u_0\in\E_\alpha $ be fixed and set
$$
w:=u(\cdot;u_0)-\phi \,,\qquad w_0:=u_0-\phi\,,
$$
where $u(\cdot;u_0)\in C\big(I(u_0),\E_\alpha\big)$ is the maximal solution to~\eqref{P} provided by \Cref{T1}. 
Then $w\in C\big(I(u_0),\E_\alpha\big)$, and it follows from \Cref{phi} and \eqref{u} that 
  \begin{equation}\label{14}
	w(t,a)\, =\, \left\{ \begin{aligned}
		&\Pi(a,a-t)\, w_0(a-t) + G_{F(u)-F(\phi)}(t,a)\, ,& &   (t,a)\in I(u_0)\times J\,,\ t\le a\, ,\\
		& \Pi(a,0)\, \big(B_u(t-a)-\phi(0)\big)+G_{F(u)-F(\phi)}(t,a)\, ,& &  (t,a)\in I(u_0)\times J\, ,\ t>a\,.
	\end{aligned}
	\right.
\end{equation}
We next use the linearizations for $F$ and $B_u$. To this end, we note from~\eqref{A1d} that
$$
\|\bar v-\bar\phi\|_{E_\alpha} \le \|\nu\|_{L_\infty(J,\ml(E_\alpha))} \,\| v-\phi\|_{\E_\alpha} 
$$
for $v\in\E_\alpha$ so that, using~\eqref{B2}, we can write
%	$$
%	b(\bar v,\cdot)=b(\bar \phi,\cdot)+\partial b(\bar\phi,\cdot)[\bar v-\bar \phi] +R_b(\bar v-\bar \phi)
%	$$
%	with reminder term
%	$$
%	\|R_b(\bar v-\bar \phi)\|_{L_\infty(J,\ml(E_\alpha,E_0))}=o\big(\|\bar v-\bar \phi\|_{E_\alpha}\big)\ \text { %as }\ \|\bar v-\bar \phi\|_{E_\alpha}\to 0\,.
%	$$
%	Since $b$ is continuously differentiable, this yields
	\begin{subequations}\label{14E}
		\begin{equation}\label{14Ea}
			b(\bar v,\cdot)v-b(\bar \phi,\cdot)\phi=b(\bar\phi,\cdot)( v- \phi) +\partial b(\bar\phi,\cdot)[\bar v-\bar \phi]\phi+ R_b( v- \phi)
		\end{equation}
		with reminder term
		\begin{equation}\label{14E2}
			\|R_b( v-\phi )\|_{\E_0}=o\big(\| v- \phi\|_{\E_\alpha}\big)\ \text { as }\ \| v- \phi\|_{\E_\alpha}\to 0\,.
		\end{equation}
Also note that~\eqref{B2} entails
\begin{equation}\label{14C}
	\|\partial b(\bar\phi,a)[\bar v]\phi(a)\|_{E_0}\le c_b\,\|\bar v\|_{E_\alpha}\,\|\phi(a)\|_{E_0}\,,\quad a\in J\,,\quad \bar v\in E_\alpha\,,
\end{equation}
\end{subequations}
with $c_b:=\|\partial b(\bar\phi,\cdot)\|_{\ml(E_\alpha,L_{\infty}(J,\ml(E_\alpha,E_0)))}$. 
	Similarly, due to~\eqref{B1}, $F:\E_\alpha\to\E_0$ is differentiable at~$\phi$ and
	\begin{subequations}\label{14G}
		\begin{equation}\label{14Ga}
			F(v)=F(\phi)+\partial F(\phi)(v-\phi)+R_F(v-\phi)
		\end{equation}
		with
		\begin{equation}\label{14G2}
			\partial F(\phi)\in\ml(\E_\alpha,\E_0)\,,\qquad 	\|R_F(v-\phi)\|_{\E_0}=o\big(\| v- \phi\|_{\E_\alpha}\big)\ \text { as }\ \|v-\phi\|_{\E_\alpha}\to 0\,.
		\end{equation}
	\end{subequations}
In particular, 
\begin{equation}\label{15}
	G_{F(u)-F(\phi)}=G_{\partial F(\phi)w+R_F(w)}\,.
\end{equation}	
Recalling~\eqref{500b} we set (slightly abusing notation)
\begin{align*}
		B_w(t):& =B_u(t)-\phi(0)\,  =\, \int_0^{a_m}  \big[b(\bar u(t),a)\, u(t,a)-b(\bar \phi,a)\phi(a)\big]\, \rd
		a \nonumber\\
		&= \int_0^{a_m}  b(\bar \phi,a)\, w(t,a)\, \rd a + \int_0^{a_m}  \partial b(\bar \phi,a)[\bar w(t)]\, \phi(a)\,\rd a+\int_0^{a_m}  R_b(w(t)) (a)\, \rd a \\
&= \int_0^{a_m}  b(\bar \phi,a)\, w(t,a)\, \rd a + \int_0^{a_m}   \int_0^{a_m}  \partial b(\bar \phi,\sigma)[\nu(a) w(t,a)]\, \phi(\sigma)\,\rd \sigma \,\rd a\\
&\quad +\int_0^{a_m}  R_b(w(t))(a)\, \rd a
	\end{align*}
for $t\in I(u_0)$, where we used~\eqref{14Ea} and the linearity of $\partial b(\bar \phi,\sigma)[\cdot]$ for the third respectively fourth equality.  Introducing (see \eqref{A1d}, \eqref{B3}, \eqref{B2b})
\begin{subequations}\label{bb}
	\begin{equation}
	\mathfrak{b}_\phi \in   L_{1}\big(J,\ml(E_\theta)\big)\cap L_{\infty}\big(J,\ml(E_\theta)\big)\,,\quad \theta\in\{0,\alpha,\vartheta\}\,,
\end{equation}
by
	\begin{equation}
	\mathfrak{b}_\phi (a)v:= b(\bar \phi,a) v+\int_0^{a_m}  \partial b(\bar \phi,\sigma)[\nu(a) v]\, \phi(\sigma)\,\rd \sigma\,,\quad a\in J\,,\quad v\in E_0\,,
\end{equation}
\end{subequations}
we obtain
\begin{align}
	B_w(t)& = \int_0^{a_m}  \mathfrak{b}_\phi (a)\, w(t,a)\, \rd a+\int_0^{a_m}  R_b(w(t))(a) \, \rd a\,,\quad t\in I(u_0)\,. \label{16}
\end{align}
Consequently, we infer from~\eqref{14}, \eqref{15}, and \eqref{16} the following intermediate result:

\begin{prop}\label{intermediate}
Let $u_0\in \E_\alpha$ and let $u(\cdot;u_0)\in C\big(I(u_0),\E_\alpha\big)$ be the maximal solution to~\eqref{P}. If
$w=u(\cdot;u_0)-\phi$ and $w_0=u_0-\phi$, 
then $w\in C(I(u_0),\E_\alpha)$
satisfies
 	\begin{subequations}\label{17}
  \begin{equation}\label{17A}
	w(t,a)\, =\, \left\{ \begin{aligned}
		&\Pi(a,a-t)\, w_0(a-t) + G_{\partial F(\phi)w+R_F(w)}(t,a)\, ,& &   (t,a)\in I(u_0)\times J\,,\ t\le a\, ,\\
		& \Pi(a,0)\, B_w(t-a)+ G_{\partial F(\phi)w+R_F(w)}(t,a)\, ,& &  (t,a)\in I(u_0)\times J\, ,\ t>a\, ,
	\end{aligned}
	\right.
\end{equation}
where $B_w$ satisfies \eqref{16}, i.e.
\begin{equation}\label{17B}
	\begin{split}
		B_w(t)& =\mathcal{M}_\phi\big(w(t)\big)+h_w(t)\,,\quad t\in I(u_0)\,,
	\end{split}
\end{equation}
with $\mathcal{M}_\phi$ being defined in \eqref{17C} and
\begin{equation}\label{17Bb}
h_w(t):=\int_0^{a_m}  R_b( w(t))(a) \, \rd a\,, \quad t\in I(u_0)\,.
\end{equation}
\end{subequations}
\end{prop}

It is worth pointing out that $w$ is thus the (generalized) solution to
	\begin{align*}
		\partial_t w+ \partial_aw \, &=     A(a)w +\partial F(\phi)w+R_F(w)\,, && t\in I(u_0)\, ,\quad a\in J\, ,\\ 
		w(t,0)&=\mathcal{M}_\phi(w(t,\cdot)) +h_w(t) \,, && t\in I(u_0)\,, \\
		w(0,a)&=  w_0(a)\,, && a\in J\,,
	\end{align*} 
where $\mathcal{M}_\phi$ from \eqref{17C} is the linearization of the right-hand side of the age boundary condition~\eqref{P2}.
The proof of \Cref{T2} is then based on suitable estimates on $w$ given by~\eqref{17}. To this end it is instrumental to investigate first the linear counterpart of~\eqref{17} more generally. This is the purpose of the next section. Regarding the data note that 
\begin{equation*}
	\partial F(\phi)w+R_F(w)\in C\big([0,T],\E_0\big)\,,\quad h_w\in C\big([0,T],E_0\big)\,,\quad w_0\in\E_\alpha\,.
\end{equation*}
We then shall return to~\Cref{intermediate} and continue from there in Section~\ref{Sec:6}.

%%%%%%%%%%%%%%%%%%%%%%%%%%%%%%%%%%%%%%%%%%%%%%%%%%%%%%%%%%%%%%%
%%%%%%%%%%%%%%%%%%%%%%%%%%%%%%%%%%%%%%%%%%%%%%%%%%%%%%%%%%%%%%%

\section{The Linearized Problem}\label{Sec:Linear}

%%%%%%%%%%%%%%%%%%%%%%%%%%%%%%%%%%%%%%%%%%%%%%%%%%%%%%%%%%%%%%%
%%%%%%%%%%%%%%%%%%%%%%%%%%%%%%%%%%%%%%%%%%%%%%%%%%%%%%%%%%%%%%%

As just announced it is appropriate at this stage to consider the linear version of~\eqref{17}. More precisely, given $T>0$, we fix
\begin{equation}\label{99}
f\in C\big([0,T],\E_0\big)\,,\quad h\in C\big([0,T],E_0\big)\,,\quad z \in\E_0\,,\quad \gamma\in\R\,,
\end{equation}
and set
\begin{equation}\label{4997}
\Pi_\gamma(a,\sigma):=e^{-\gamma(a-\sigma)}\Pi(a,\sigma)\,,\qquad a\in J\,,\quad 0\le \sigma\le a \,,
\end{equation}
and
\begin{subequations}\label{1000}
\begin{equation}\label{4999}
G_{f}^\gamma(t,a):=\int_{(t-a)_+}^{t}\Pi_\gamma(a,a-t+s)\, f(s,a-t+s)\,\rd s\,,\qquad a\in J\,,\quad t\in [0,T]\,.
\end{equation}
The additional parameter $\gamma$ is introduced for technical reasons, its role will become clear later in Section~\ref{Sec:6} (see~\eqref{2i3} for definiteness). We then define in dependence on these data the function $W=W_{z ,f}^{\gamma,h}$  by
	\begin{equation}\label{4998}
		W_{z ,f}^{\gamma,h}(t,a)\, :=\, \left\{ \begin{aligned}
			&\Pi_\gamma(a,a-t)\, z (a-t) + G_{f}^\gamma(t,a)\, ,& &   (t,a)\in [0,T]\times J\,,\  t\le a\, ,\\
			& \Pi_\gamma(a,0)\, B_{z ,f}^{\gamma,h}(t-a)+ G_{f}^\gamma(t,a)\, ,& &  (t,a)\in [0,T]\times J\, ,\ t>a\, ,
		\end{aligned}
		\right.
	\end{equation}
where $B=B_{z ,f}^{\gamma,h}$ satisfies
\begin{align}\label{5000}
		B(t)\,  =\, & \int_0^t  \mathfrak{b}_\phi(a)\, \Pi_\gamma(a,0)\, B(t-a)\, \rd
			a\, +\int_t^{a_m}  \mathfrak{b}_\phi(a)\, \Pi_\gamma(a,a-t)\, z (a-t)\, \rd a\,\nonumber\\
			&  +\int_0^{a_m}  \mathfrak{b}_\phi(a)\, G_{f}^\gamma(t,a)\, \rd a + h(t) 
	\end{align}
\end{subequations}
with the understanding in the following that $\mathfrak{b}_\phi(a)=0$ whenever $a\notin J$. That is,
	\begin{align}\label{key}
B_{z ,f}^{\gamma,h}(t)=\mathcal{M}_\phi\big(W_{z ,f}^{\gamma,h}(t,\cdot)\big) +h(t)\,,\quad t\in [0,T]\,.
	\end{align}
Let us point out that $W=W_{z ,f}^{\gamma,h}$ represents the (generalized) solution  to the linear problem 
	\begin{align*}
		\partial_t W+ \partial_aW \, &=     \big(-\gamma+A(a)\big) W +f(t,a)\,, && t\in [0,T]\, ,\quad a\in J\, ,\\ 
		W(t,0)&=\mathcal{M}_\phi\big(W(t,\cdot)\big) +h(t) \,, && t\in [0,T]\,, \\
		W(0,a)&=  z (a)\,, && a\in J\,,
	\end{align*} 
%with
%$$
%A_\gamma(a):=-\gamma+A(a)\,,\quad a\in J\,,
%$$
and is formally obtained by an integration along characteristics. The subsequent auxiliary results are considerably easier to derive on the formal level of this differential equation. 

The linear structure of \eqref{1000} ensures the superposition
\begin{equation}\label{W-1}
	W_{z ,f}^{\gamma,h}=W_{z ,0}^{\gamma,0}+W_{0,f}^{\gamma,0}+W_{0,0}^{\gamma,h}\,.
\end{equation}

The aim now is to give a semigroup based representation formula for $W_{z ,f}^{\gamma,h}$ related to the data $(z ,f,\gamma,h)$ which we then shall exploit for the nonlinear problem~\eqref{17}.

%%%%%%%%%%%%%%%%%%%%%%%%%%%%%%%%%%%%%%%%%%%%%%%%%%%%%%%
\subsection*{The Linearized Age Boundary Operator}
%%%%%%%%%%%%%%%%%%%%%%%%%%%%%%%%%%%%%%%%%%%%%%%%%%%%%%%

Starting  with $B_{z ,f}^{\gamma,h}$ we show, in particular, that it is well-defined and collect further properties in the next lemma.

\begin{lem}\label{L1}
 Suppose~\eqref{99}. There is a unique 
 $B=B_{z ,f}^{\gamma,h}\in C([0,T],E_0)$
 satisfying~\eqref{5000}. Moreover, it decomposes as
\begin{equation}\label{Bg}
 B_{z ,f}^{\gamma,h}(t)=B_{z ,0}^{\gamma,0}(t)+B_{0,f}^{\gamma,0} (t)+B_{0,0}^{\gamma,h}(t)\,,\quad t\in [0,T]\,,
\end{equation}
with
\begin{equation}\label{72}
	B_{z ,0}^{\gamma,0}(t)=e^{-\gamma t}\,B_{z ,0}^{0,0}(t)\,,\quad t\ge 0\,,
\end{equation}
and
\begin{equation}\label{72C}
	B_{0,f}^{\gamma,0} (t)=\int_0^t B_{f(s),0}^{\gamma,0} (t-s) \, \rd s \,,\quad t\in [0,T]\,.
\end{equation}
\end{lem}

\begin{proof}
It follows as in \eqref{o9}  that
$$
[f\mapsto G_{f}^\gamma ]\in \ml\big(C\big([0,T],\E_0\big),C\big([0,T],\E_\alpha\big)\big)\,. 
$$
Setting
$$
\tilde h(t):=\mathcal{M}_\phi\big(G_{f}^\gamma(t,\cdot)\big)+ h(t)\,,\quad t\in [0,T]\,,
$$
we thus obtain $\tilde h\in C([0,T],E_0)$ due to \eqref{99} and \eqref{17Cx}. Therefore,  \cite[Lemma 6.1]{WalkerIUMJ} along with \eqref{A1a}, \eqref{EO}, \eqref{A4}, \eqref{A1d}, \eqref{B3}, and \eqref{B2b} entails that there exists a unique function \mbox{$B=B_{z ,f}^{\gamma,h}\in C([0,T],E_0)$} satisfying~\eqref{5000} and
	\begin{equation}\label{ml}
	\big[z \to B_{z ,0}^{\gamma,0}\big]\in \ml\big(\E_0,C([0,T],E_0)\big)\,.
\end{equation}
 The linear structure of~\eqref{5000}  ensures~\eqref{Bg}. As for~\eqref{72} note first that $T>0$ can be chosen arbitrary if $f$ and $h$ are zero. From~\eqref{5000} and~\eqref{4997} we derive
\begin{align*}
e^{\gamma t}B_{z ,0}^{\gamma,0} (t)\,  =\, \int_0^t  \mathfrak{b}_\phi(a)\, \Pi(a,0)\, e^{\gamma (t-a)}\, B_{z ,0}^{\gamma,0}(t-a)\, \rd
	a +\int_t^{a_m}  \mathfrak{b}_\phi(a)\, \Pi(a,a-t)\, z (a-t)\, \rd a
\end{align*}
for $t\ge 0$. That is, $t\mapsto e^{\gamma t}B_{z ,0}^{\gamma,0}(t)$ satisfies the same equation as $B_{z ,0}^{0,0}$. Uniqueness implies then~\eqref{72}. 

With the same idea we prove~\eqref{72C}. To this end, we first note that the integral in~\eqref{72C} is well defined, since 
$$
\big[s \mapsto B_{f(s),0}^{\gamma,0}(t-s)\big]\in C\big([0,t],E_0\big)\,,\quad t\in [0,T]\,,
$$ 
as is easily seen by the triangle inequality together with~\eqref{ml} and the assumption $f\in C\big([0,T],\E_0\big)$. Next, by~\eqref{5000} we have
\begin{align*}
B_{f(s),0}^{\gamma,0}(t-s)  =\, & \int_0^{t-s}  \mathfrak{b}_\phi(a)\, \Pi_\gamma(a,0)\,  B_{f(s) ,0}^{\gamma,0}(t-s-a)\, \rd a\\
& +\int_{t-s}^{a_m}  \mathfrak{b}_\phi(a)\, \Pi_\gamma(a,a-t+s)\, f(s,a-t+s)\, \rd a 
\end{align*}
so that
\begin{align*}
	\int_0^t B_{f(s),0}^{\gamma,0}(t-s) \, \rd s
	=\, & \int_0^t \int_0^{t-s}  \mathfrak{b}_\phi(a)\, \Pi_\gamma(a,0)\,  B_{f(s) ,0}^{\gamma,0}(t-s-a)\, \rd a\, \rd s\\
	& +\int_0^t\int_{t-s}^{a_m}  \mathfrak{b}_\phi(a)\, \Pi_\gamma(a,a-t+s)\, f(s,a-t+s)\, \rd a\, \rd s 
\end{align*}
for $t\in [0,T]$.
Therefore, applying Fubini's theorem, we derive
\begin{align*}
	\int_0^t B_{f(s),0}^{\gamma,0}(t-s) \, \rd s
  =\, & \int_0^t  \mathfrak{b}_\phi(a)\, \Pi_\gamma(a,0)\, \left(\int_0^{t-a}  B_{f(s) ,0}^{\gamma,0}(t-a-s)\, \rd s\right) \rd a\\
		& +\int_0^{a_m}\mathfrak{b}_\phi(a)\int_{(t-a)_+}^{t}   \Pi_\gamma(a,a-t+s)\, f(s,a-t+s)\, \rd s\, \rd a 
\end{align*}
Consequently, recalling~\eqref{5000} and~\eqref{4999}, we see  that $t \mapsto \int_0^t B_{f(s)}^{\gamma,0,0} (t-s) \, \rd s$ satisfies the same equation as $B_{0,f}^{\gamma,0}$ so that~\eqref{72C} follows by uniqueness.
\end{proof}

We next derive an estimate on $B_{0,0}^{\gamma,h}$.

\begin{lem}\label{LemmaFF}
	Let $h\in C\big([0,T],E_0\big)$ and $\gamma\in\R$. Then there are constants $\mu=\mu(\alpha,b,\phi)>0$ and $c_1=c_1(\alpha,b,\phi)>0$ (both independent of $\gamma$ and $h$) such that
	\begin{align}\label{70}
		\|B_{0,0}^{\gamma,h}(t)\|_{E_0} 
		\le  c_1 \int_0^t (t-a)^{-\alpha}  \,  e^{(\mu+\varpi-\gamma) (t-a)}\, \|h(a)\|_{E_0}\, \rd a\, + \|h(t)\|_{E_0}\,,\quad t\in [0,T]\,.
	\end{align}
\end{lem}

\begin{proof}
	We use~\eqref{5000}, \eqref{EO}, ~\eqref{A1b} and~\eqref{14C} to get 
\begin{align*}
	\|B_{0,0}^{\gamma,h}(t)\|_{E_0}  \le\, & \int_0^t  \|\mathfrak{b}_\phi(a)\|_{\ml(E_\alpha,E_0)}\, \|\Pi_\gamma(a,0)\|_{\ml(E_0,E_\alpha)}\, \|B_{0,0}^{\gamma,h}(t-a)\|_{E_0}\, \rd
	a\, +\|h(t)\|_{E_0}\\
	\le \, & \|\mathfrak{b}_\phi\|_{L_\infty(J,\ml(E_\alpha,E_0))}\, M_\alpha \int_0^t  e^{(\varpi-\gamma) a}\, a^{-\alpha}  \, \|B_{0,0}^{\gamma,h}(t-a)\|_{E_0}\, \rd a\,+\| h(t)\|_{E_0}	
\end{align*}
for $t\in [0,T]$. That is, there is some $c=c(\alpha,b,\phi)>0$ such that
\begin{align*}
	e^{-(\varpi-\gamma) t}\|B_{0,0}^{\gamma,h}(t)\|_{E_0} 
	\le  c \int_0^t (t-a)^{-\alpha}  \,  e^{-(\varpi-\gamma) a}\, \|B_{0,0}^{\gamma,h}(a)\|_{E_0}\, \rd a\, +e^{-(\varpi-\gamma) t}\|h(t)\|_{E_0}
\end{align*}
for $t\in [0,T]$. Gronwall's inequality~\cite[Lemma 7.1.1]{Henry} now implies~\eqref{70} for some constants \mbox{$\mu=\mu(\alpha,b,\phi)>0$} and $c_1=c_1(\alpha,b,\phi)>0$.
\end{proof}

%%%%%%%%%%%%%%%%%%%%%%%%%%%%%%%%%%%%%%%%%%%%%%%%%%%%%%%
\subsection*{The Linear Part and its Associated Semigroup}
%%%%%%%%%%%%%%%%%%%%%%%%%%%%%%%%%%%%%%%%%%%%%%%%%%%%%%%

Now that $B_{z,f}^{\gamma,h}$ is well-defined we shall focus on $W_{z,f}^{\gamma,h}$. We begin with the linear part $W_{z ,0}^{\gamma,0}$ (with vanishing $f$ and $h$). We first note that it defines a strongly continuous semigroup on $\E_0$ and, due to the regularizing effects of the evolution operator $\Pi$, also on $\E_\alpha$. This semigroup was investigated in \cite{WalkerIUMJ}. We recall the main results and add some other useful properties.
 
\begin{prop}\label{LemmaBB}
	Set 
	$$
	\mS (t)z :=W_{z ,0}^{0,0}(t,\cdot)\,,\qquad t\ge 0\,,\quad z\in\E_0\,.
	$$ 
	Then $(\mS(t))_{t\ge 0}$ is a strongly continuous semigroup on $\E_0$ and (its restriction) also on $\E_\alpha$ with
	\begin{equation}\label{E3o}
		\|\mS(t)\|_{\ml(\E_0,\E_\alpha)}\le C_\alpha\, t^{-\alpha} e^{\varsigma_\alpha t}\,,\quad t> 0\,,
	\end{equation}
	for some $C_\alpha\ge 1$ and $\varsigma_\alpha\in\R$. Moreover,
	\begin{equation}\label{LemmaG}
		W_{z ,0}^{\gamma,0}(t,\cdot)=e^{-\gamma t}\,\mS (t)z\,,\qquad t\ge 0\,,\quad z\in\E_0\,,
	\end{equation}
	and, for $z\in\E_0$ and $f\in C\big([0,T],\E_0\big)$ with $T>0$,
	\begin{equation}\label{W2}
	W_{z ,f}^{\gamma,0}(t,\cdot)=e^{-\gamma t}\mS (t)z +\int_0^t e^{-\gamma (t-s)}\mS (t-s) f(s)\,\rd s\,,\quad t\in [0,T]\,.
	\end{equation}
\end{prop}

\begin{proof}
It follows from assumptions \eqref{A1a}, \eqref{EO}, \eqref{A4}, \eqref{A1d}, \eqref{B3}, \eqref{B2b} together with \cite[Theorem 2.8]{WalkerIUMJ}  that $(\mS(t))_{t\ge 0}$ defines a strongly continuous semigroup on $\E_0$ and on $\E_\alpha$ satisfying~\eqref{E3o}. Identity~\eqref{LemmaG} is easily derived from the definition of $W_{z ,0}^{\gamma,0}$ in~\eqref{4998} along with~\eqref{4997} and~\eqref{72} from \Cref{L1}.
As noted in~\eqref{W-1}, the linear structure of~\eqref{1000} entails
		\begin{equation*}
			W_{z ,f}^{\gamma,0}=W_{z ,0}^{\gamma,0}+W_{0,f}^{\gamma,0}
		\end{equation*}
so that, due to~\eqref{LemmaG},  identity~\eqref{W2} will follow once we have identified the integral term therein as $W_{0,f}^{\gamma,0}$. To this end observe that~\eqref{LemmaG} ensures
\begin{equation}\label{just}
\int_0^t e^{-\gamma (t-s)}\mS (t-s) f(s)\,\rd s	=\int_0^t W_{f(s),0}^{\gamma,0}(t-s,\cdot)\,\rd s	\,,\quad t\in [0,T]\,.
\end{equation}
Let $t\in [0,T]$ and $a\in J$. If $t \leq a$, then $t-s \leq a-s \leq a$ for $s\in [0,t]$ and therefore, by~\eqref{just}, ~\eqref{4998}, and~\eqref{4999}, we indeed have
\begin{align*}
	\int_0^t \left( e^{-\gamma (t-s)}\mS (t-s) f(s) \right) (a)\, \rd s &= \int_0^t \Pi_{\gamma} (a,a-t+s) f(s,a-t+s) \, \rd s\\
	&=G_{f}^\gamma(t,a)=W_{0,f}^{\gamma,0}(t,a)
\end{align*}
in this case. Consider then $a<t$. In this case we infer from ~\eqref{just} and~\eqref{4998} that
\begin{align*}
	\int_0^t  \left( e^{-\gamma (t-s)}\mS (t-s) f(s) \right) (a)\, \rd s =\, & \Pi_{\gamma} (a,0)\int_0^{t-a}  B_{f(s),0}^{\gamma,0} (t-a-s) \, \rd s\\
	&+ \int_{t-a}^t \Pi_{\gamma} (a,s+a-t) f(s,a-t+t) \, \rd s\,.
\end{align*}
From~\eqref{72C} in \Cref{L1} and~\eqref{4999} we then obtain
\begin{align*}
	\int_0^t  \left( e^{-\gamma (t-s)}\mS (t-s) f(s) \right) (a)\, \rd s =\, & \Pi_{\gamma} (a,0) B_{0,f}^{\gamma,0} (t-a) +G_{f}^\gamma(t,a)=W_{0,f}^{\gamma,0}(t,a) \,.
\end{align*}
Consequently, we have
$$
\int_0^t e^{-\gamma (t-s)}\mS (t-s) f(s)\,\rd s=W_{0 ,f}^{\gamma,0}(t,\cdot)\,,\quad t\in [0,T]\,,
$$
which proves \eqref{W2}.
\end{proof}

As remarked previously the strongly continuous semigroup  $(\mS(t))_{t\ge 0}$ was investigated in
\cite{WalkerIUMJ}. In particular, it was shown therein that the regularizing effect stated in~\eqref{E3o} implies that its generator can be perturbed by an operator belonging to $\ml(\E_\alpha,\E_0)$ and still yields a strongly continuous semigroup on $\E_0$. We use this observation now to derive the following representation for the solution of the perturbed  Cauchy problem. So far, we refrained from indicating the dependence of $\mS(t)$ on the fixed equilibrium  $\phi$. For later use, however, we indicate this dependence in the notation of the semigroup associated with the perturbation.

\begin{prop}\label{LemmaH}
	Let $\A$ denote the infinitesimal generator of the strongly continuous semigroup $(\mS(t))_{t\ge 0}$ on $\E_0$ introduced in \Cref{LemmaBB} and consider $\B:=\partial F(\phi)\in \ml(\E_\alpha,\E_0)$. Then $\A+\B$ generates a strongly continuous semigroup $(\T_\phi(t))_{t\ge 0}$ on $\E_0$ and also on $\E_\alpha$. Moreover, there are $N_\alpha:=N_\alpha(\phi)\ge 1$ and $\omega_\alpha:=\omega_\alpha(\phi)\in\R$ such that
	\begin{equation}\label{E3ooo}
	\|\T_\phi(t)\|_{\ml(\E_\alpha)}+t^{\alpha}\,	\|\T_\phi(t)\|_{\ml(\E_0,\E_\alpha)}\le N_\alpha\,  e^{ -\omega_\alpha t}\,,\quad t\ge  0\,.
	\end{equation}
Let  $\gamma\in\R$,  $v_0\in \E_\alpha$, and $g\in C([0,T],\E_0)$. If $v\in C([0,T],\E_\alpha)$ solves
	\begin{align*}
	v(t)=e^{-\gamma t}\mS(t)v_0+\int_0^t e^{-\gamma (t-s)}\mS(t-s)\,\big((\gamma+\B) v(s)+g(s)\big)\, \rd s\,,\quad t\in [0,T]\,,
	\end{align*}
	then
	\begin{align*}%\label{H2}
		v(t)=\T_\phi(t)v_0+\int_0^t \T_\phi(t-s)\, g(s) \, \rd s\,,\quad t\in [0,T]\,.
	\end{align*}
\end{prop}

\begin{proof}
It follows from \cite[Theorem~2.8]{WalkerIUMJ} that $\A+\B$ with domain $D(\A+\B)=D(\A)$ (in particular, we have $D(\A)\hookrightarrow \E_\alpha$) generates a strongly continuous semigroup $(\T_\phi(t))_{t\ge 0}$ on $\E_0$ and on $\E_\alpha$ satisfying~\eqref{E3ooo}.
%\begin{equation*}\label{E3oo}
%\|\T_\phi(t)\|_{\ml(\E_\alpha)}+	t^{\beta}\|\T_\phi(t)\|_{\ml(\E_0,\E_\beta)}\le N_\beta\, e^{-\omega_\beta %t}\,,\quad t\ge  0\,,
%\end{equation}
%for $\beta\in\{0,\alpha\}$ and some $N_\beta\ge 1$ and $\omega_\beta\in\R$. 
The remainder of the proof is now the same as in \cite[Proposition~4.17]{WebbBook} except that $\B$ herein is no  bounded perturbation on $\E_0$. We thus include the details here. 

{{\bf (i)}} Let first $g\in C^1([0,T],\E_0)$ and  $v_0\in D(\A)$. Define
	\begin{align}\label{74}
	x(t):=\T_\phi(t)v_0+\int_0^t \T_\phi(t-s) g(s)\, \rd s\,,\quad t\in [0,T]\,,
	\end{align}
and note that $x\in C^1([0,T],\E_0)\cap C([0,T],D(\A))$ is the unique solution to
$$
x'=(-\gamma+\A) x+(\gamma+\B) x+ g(t)\,,\quad t\in [0,T]\,,\qquad x(0)=v_0\,.
$$
Hence,
\begin{align*}
	x(t)=e^{-\gamma t}\mS(t)v_0+\int_0^t e^{-\gamma (t-s)}\mS(t-s) \big((\gamma+\B) x(s)+g(s)\big)\, \rd s\,,\quad t\in [0,T]\,,
\end{align*}
so that, using~\eqref{E3o},
\begin{align*}
	\|x(t)-v(t)\|_{\E_\alpha}&\le \int_0^t e^{-\gamma (t-s)}\,\|\mS(t-s)\|_{\ml( \E_0,\E_\alpha)} \, \|\gamma+\B\|_{\ml(\E_\alpha,\E_0)} \,  \|x(s)-v(s)\|_{\E_\alpha}\, \rd s \\
	&\le C_\alpha\, \|\gamma+\B\|_{\ml(\E_\alpha,\E_0)}  \int_0^t (t-s)^{-\alpha} \, e^{(\varsigma_\alpha-\gamma) (t-s)}\,   \|x(s)-v(s)\|_{\E_\alpha}\, \rd s
\end{align*}
for $t\in [0,T]$. Gronwall's inequality \cite[II.Theorem 3.3.1]{LQPP} implies that indeed $v=x$ on $[0,T]$.

{{\bf (ii)}} Consider now $g_k\in C^1([0,T],\E_0)$ and  $v_{0,k}\in D(\A)$ with 
$$
g_k\to g\ \text{ in }\ C([0,T],\E_0)\,,\qquad v_{0,k}\to v_0\ \text{ in }\ \E_0
$$
as $k\to\infty$. 
Define $x$ again by \eqref{74} and accordingly
\begin{align*}
	x_k(t):=\T_\phi(t)v_{0,k}+\int_0^t \T_\phi(t-s) g_k(s)\, \rd s\,,\quad t\in [0,T]\,.
\end{align*}
Taking $\beta\in\{0,\alpha\}$ and invoking~\eqref{E3ooo} we obtain
\begin{align*}
	t^{\beta}\, \|x(t)-x_k(t)\|_{\E_\beta}& \le t^{\beta}\,\|\T_\phi(t)\|_{\ml( \E_0,\E_\beta)} \, \|v_0-v_{0,k}\|_{\E_0}\\
	&\quad + t^{\beta}\int_0^t \|\T_\phi(t-s)\|_{\ml( \E_0,\E_\beta)} \,  \|g(s)-g_k(s)\|_{\E_0}\, \rd s\\
	&\le  c(T) \, \|v_0-v_{0,k}\|_{\E_0} +  c(T) \|g-g_k\|_{C([0,T],\E_0)}
\end{align*}
for $t\in [0,T]$. Consequently, as $k\to \infty$,
\begin{equation}\label{75}
	t^\beta\,\|x(t)-x_k(t)\|_{\E_\beta}\to 0\ \text{ uniformly with respect to $t\in [0,T]$}\,,\qquad \beta\in\{0,\alpha\}\,.
\end{equation}
Since
\begin{align*}
	x_k(t)=e^{-\gamma t}\mS(t)v_{0,k}+\int_0^t e^{-\gamma (t-s)}\mS(t-s) \,\big((\gamma+\B) x_k(s)+g_k(s)\big)\, \rd s\,,\quad t\in [0,T]\,,
\end{align*}
according to {{\bf (i)}}, it thus follows from~\eqref{75} (using~\eqref{E3o}) that
\begin{align}\label{78}
	x(t)=e^{-\gamma t}\mS(t)v_{0}+\int_0^t e^{-\gamma (t-s)}\mS(t-s) \,\big((\gamma+\B) x(s)+g(s)\big)\, \rd s\,,\quad t\in [0,T]\,.
\end{align}
Consequently, ~\eqref{78} and~\eqref{E3o} entail 
\begin{align*}
	\|x(t)-v(t)\|_{\E_\alpha}& \le  \int_0^t  e^{-\gamma(t-s)}\,\|\mS(t-s)\|_{\ml( \E_0,\E_\alpha)} \,  \|\gamma+\B\|_{\ml(\E_\alpha,\E_0)}\,  \|x(s)-v(s)\|_{\E_\alpha}\, \rd s\\
	&\le  c(T)\,\int_0^t (t-s)^{-\alpha} \,  \|x(s)-v(s)\|_{\E_\alpha}\, \rd s
\end{align*}
for  $t\in [0,T]$; that is $v=x$ on $[0,T]$ by Gronwall's inequality.
\end{proof}

Of course, \Cref{LemmaH} is not restricted to the particular choice of $\partial F(\phi)$ for the perturbation~$\B\in \ml(\E_\alpha,\E_0)$.

\begin{rem}
	It is worth emphasizing that the strongly continuous semigroup $(\T_\phi(t))_{t\ge 0}$ is the solution operator associated with the linearization (see~\eqref{Plinear}) of ~\eqref{P} given by
	\begin{align*}
		\partial_t v+ \partial_av \, &=     A(a)v +\partial F(\phi) v\,, \qquad t>0\, ,\quad a\in (0,a_m)\, ,\\ 
		v(t,0)&=\mathcal{M}_\phi(v(t)) \,, \qquad t>0\, , \\
		v(0,a)&=  v_0(a)\,, \qquad a\in (0,a_m)\,,
	\end{align*} 
that is, $v(t)=\T_\phi(t)v_0$, $t\ge 0$, defines the unique mild solution for each $v_0\in\E_0$. See~\cite{WalkerIUMJ}.
\end{rem}

The assumption $\omega_\alpha(\phi)>0$ in \eqref{E3ooo} corresponds to an exponential decay of the semigroup $(\T_\phi(t))_{t\ge 0}$ and plays an important role in the subsequent stability analysis. We thus add some comments on this issue.

\begin{rem}\label{R14}
{\bf (a)} Assuming $\omega_\alpha(\phi)>0$ in \eqref{E3ooo} is equivalent to assuming that the growth bound of the semigroup $(\T_\phi(t))_{t\ge 0}$ on $\E_\alpha$ is negative.

\begin{proof}
Let $\omega(\T_\phi)$ be the growth bound of the semigroup $(\T_\phi(t))_{t\ge 0}$ on $\E_\alpha$. Clearly, \eqref{E3ooo} yields that $\omega(\T_\phi)\le-\omega_\alpha(\phi)$. Assume now that $\omega(\T_\phi)<0$. Then, for $\omega(\T_\phi)< -(\omega+\ve)<-\omega<0$ there is $N\ge 1$ such that
	\begin{equation*}
	\|\T_\phi(t)\|_{\ml(\E_\alpha)}\le N\,  e^{ -(\omega+\ve) t}\,,\quad t\ge  0\,.
\end{equation*}
This along with \eqref{E3ooo} implies
\begin{equation*}
	\|\T_\phi(t)\|_{\ml(\E_0,\E_\alpha)}\le \|\T_\phi(t-1)\|_{\ml(\E_\alpha)}\, \|\T_\phi(1)\|_{\ml(\E_0,\E_\alpha)} \le N\,  e^{ -(\omega+\ve) (t-1)}\, N_\alpha\, e^{\vert \omega_\alpha\vert } 
	\le N_0\,  e^{ -\omega t}\, t^{-\alpha}
\end{equation*}
for $t\ge 1$ and some $N_0\ge 1$,
while such an estimate is obviously implied by \eqref{E3ooo}  for $t\in (0,1)$. Hence, $\omega_\alpha(\phi)$ can be chosen positive in~\eqref{E3ooo} if $\omega(\T_\phi)<0$.
\end{proof}

{\bf (b)} If the semigroup $(\T_\phi(t))_{t\ge 0}$  on $\E_\alpha$ is eventually compact, then the growth bound of the semigroup $(\T_\phi(t))_{t\ge 0}$ on $\E_\alpha$ coincides with the spectral bound of its generator $\A+\B$, see \cite[IV.Corollary 3.12]{EngelNagel}. In some cases (e.g. for the trivial equilibrium $\phi=0$, or if $\B=\partial F(\phi)=0$ so that $\T_\phi=\mS$) one can indeed show the eventual compactness of $(\T_\phi(t))_{t\ge 0}$ \cite{WalkerIUMJ} so that $\omega_\alpha(\phi)>0$ in \eqref{E3ooo} is equivalent to a negative spectral bound of the generator.
\end{rem}

We will get back to this point in \Cref{Sec:Exmp}.

%%%%%%%%%%%%%%%%%%%%%%%%%%%%%%%%%%%%%%%%%%%%%%%%%%%%%%%%%%%%%%%%%%%%%%%%%

\subsection*{The Nonlinear Part} We next focus on the nonlinear part $W_{0 ,0}^{\gamma,h}$ for which we shall derive an estimate.

%%%%%%%%%%%%%%%%%%%%%%%%%%%%%%%%%%%%%%%%%%%%%%%%%%%%%%%%%%%%%%%%%%%%%%%%%

\begin{lem}\label{LemmaF}
	Let $h\in C\big([0,T],E_0\big)$ and $\gamma\in\R$. Then $W_{0,0}^{\gamma,h}\in C([0,T],\E_\alpha)$ and there are constants $\mu=\mu(\alpha,b,\phi)>0$ and $c_0=c_0(\alpha,b,\phi)>0$ (both independent of $\gamma$ and $h$) such that
	\begin{equation}\label{LemmaFest}
		\|W_{0,0}^{\gamma,h}(t,\cdot)\|_{\E_\alpha}\le c_0\int_0^t e^{(\varpi+\mu-\gamma)(t-a)}\,(t-a)^{-\alpha}\,\|h(a)\|_{E_0}\,\rd a\,,\quad t\in [0,T]\,.
	\end{equation}
\end{lem} 

\begin{proof}
To prove continuity let $0\le s\le t\le T$. Then, by~\eqref{4998},~\eqref{4997} and~\eqref{EO},
\begin{align*}
\|W_{0,0}^{\gamma,h}(t,\cdot)-W_{0,0}^{\gamma,h}(s,\cdot)\|_{\E_\alpha} &\le \int_0^s \|\Pi_\gamma(a,0)\|_{\ml(E_0,E_\alpha)}\,\|B_{0,0}^{\gamma,h}(t-a)-B_{0,0}^{\gamma,h}(s-a)\|_{E_0}\,\rd a\\
&\quad + \int_s^t \|\Pi_\gamma(a,0)\|_{\ml(E_0,E_\alpha)}\,\|B_{0,0}^{\gamma,h}(t-a)\|_{E_0}\,\rd a\\
 &\le M_\alpha\int_0^s e^{(\varpi-\gamma) a} a^{-\alpha}\,\|B_{0,0}^{\gamma,h}(t-a)-B_{0,0}^{\gamma,h}(s-a)\|_{E_0}\,\rd a\\
&\quad +M_\alpha\int_s^t e^{(\varpi-\gamma) a} a^{-\alpha}\,\|B_{0,0}^{\gamma,h}(t-a)\|_{E_0}\,\rd a\,.
\end{align*}
Since $B_{0,0}^{\gamma,h}\in C([0,T],E_0)$ according to \Cref{L1}, we infer that indeed $W_{0,0}^{\gamma,h}\in C([0,T],\E_\alpha)$. 

As for~\eqref{LemmaFest} we use the corresponding estimate on $B_{0,0}^{\gamma,h}$ established in~\eqref{70}. More precisely, we use~\eqref{70} along with~\eqref{EO} and~\eqref{4997} in the definition~\eqref{4998} of $W_{0,0}^{\gamma,h}(t)$ to derive 
\begin{align*}
	\|W_{0,0}^{\gamma,h}(t,\cdot)\|_{\E_\alpha}  & \le \int_0^{t\wedge a_m}   \|\Pi_\gamma(a,0)\|_{\ml(E_0,E_\alpha)}\, \|B_{0,0}^{\gamma,h}(t-a)\|_{E_0}\, \rd a\,\\
	& \le M_\alpha \int_0^{t}   (t-a)^{-\alpha}  \,  e^{(\varpi-\gamma) (t-a)}\, \|B_{0,0}^{\gamma,h}(a)\|_{E_0}\, \rd a\,\\
	& \le M_\alpha\, c_1 \int_0^{t}   (t-a)^{-\alpha}  \,  e^{(\varpi-\gamma) (t-a)}\, \int_0^a (a-s )^{-\alpha}  \,  e^{(\mu+\varpi-\gamma) (a-s )}\, \|h(s )\|_{E_0}\, \rd s \, \rd a\,\\
	&\quad + M_\alpha\, c_1 \int_0^{t}   (t-a)^{-\alpha}  \,  e^{(\varpi-\gamma) (t-a)}\, \|h(a)\|_{E_0}\, \rd a\\
	& \le M_\alpha\, c_1 \int_0^{t}     e^{(\mu+\varpi-\gamma) (t-s )}\, \|h(s )\|_{E_0}\, \int_s ^t (t-a)^{-\alpha}  \, (a-s )^{-\alpha}  \, \rd a\, \rd s \,\\
	&\quad + M_\alpha\, c_1 \int_0^{t}   (t-a)^{-\alpha}  \,  e^{(\varpi-\gamma) (t-a)}\, \|h(a)\|_{E_0}\, \rd a
\end{align*}
for $t\in [0,T]$. Now, noticing
$$
\int_s ^t (t-a)^{-\alpha}  \, (a-s )^{-\alpha} \, \rd a =\mathsf{B}(1-\alpha,1-\alpha)\, (t-s )^{1-2\alpha}\le c_\mu e^{\mu(t-s )}\,(t-s )^{-\alpha}\,, \quad 0\le s  <t\,,
$$
with Beta function $\mathsf{B}$, we conclude
\begin{align*}
	\|W_{0,0}^{\gamma,h}(t,\cdot)\|_{\E_\alpha}  \le   c_2 \int_0^{t}   (t-a)^{-\alpha}  \,  e^{(2\mu+\varpi-\gamma)(t-a)}\,  \|h(a)\|_{E_0}\, \rd a \,,\quad t\in [0,T]\,,
\end{align*}
as claimed.
\end{proof}

Summarizing our findings regarding $W_{z ,f}^{\gamma,h}$ given in \eqref{1000} we get:

\begin{cor}\label{LemmaB19}
	Suppose~\eqref{99}. Then
	$W_{z ,f}^{\gamma,h}\in C([0,T],\E_0)$
	satisfies
	\begin{equation*}
		W_{z ,f}^{\gamma,h}(t,\cdot)=W_{z ,f}^{\gamma,0}(t,\cdot)+W_{0,0}^{\gamma,h}(t,\cdot)=e^{-\gamma t}\mS (t)z +\int_0^t e^{-\gamma (t-s)}\mS (t-s) f(s)\,\rd s +W_{0,0}^{\gamma,h}(t,\cdot) 
	\end{equation*}
for $t\in [0,T]$, where $(\mS(t))_{t\ge 0}$ is the strongly continuous semigroup on $\E_0$ (and on $\E_\alpha$) introduced in \Cref{LemmaBB}.
If $z \in\E_\alpha$, then $W_{z ,f}^{\gamma,h}\in C([0,T],\E_\alpha)$.
\end{cor}

\begin{proof}
	This now follows from \Cref{LemmaBB}, \Cref{LemmaF}, and~\eqref{W-1}.  
\end{proof}

Let us also state the following identity for $W_{z ,f}^{\gamma,h}$ that we shall use later on.

\begin{lem}\label{LemmaB}
	Suppose~\eqref{99} and set
	$W:=W_{z ,f}^{0,h}$ for abbreviation. Then
	 $W_{z ,f}^{0,h}=W_{z ,\gamma W+f}^{\gamma,h}$ for every $\gamma\in\R$.
\end{lem}

 \begin{proof}
		Let $W=W_{z ,f}^{0,h}$ in the following, fix $\gamma\in\R$, and consider $t\in [0,T]$ and $a\in J$.\\ 
		
		{\bf (i)} If $a>t$, then, using \eqref{4998}, \eqref{4999}, \eqref{4997}, and the evolution property~\eqref{evol}, we derive
		\begin{align*}
			W_{z ,\gamma W+f}^{\gamma,h}(t,a)&= \Pi_\gamma (a,a-t) z  (a-t)\\
			&\quad  + \int_0^t \Pi_\gamma (a,a-t+s) \big( \gamma \, W (s,a-t+s) +f(s,a-t+s)  \big) \, \rd s \\
			&= \Pi_\gamma (a,a-t) z  (a-t) + \gamma\int_0^t \Pi_\gamma (a,a-t+s) \Pi (a-t+s,a-t)z  (a-t)  \, \rd s \\
			&\quad  +\gamma\int_0^t  \Pi_\gamma (a,a-t+s)  \int_0^s \Pi (a-t+s,a-t+\sigma)f(\sigma,a-t+\sigma)\, \rd \sigma  \rd s \\
			&\quad  + \int_0^t  \Pi_\gamma (a,a-t+s) f(s,a-t+s) \, \rd s \\
			&= \Pi_\gamma (a,a-t) z  (a-t) + \Pi (a,a-t)z  (a-t) \left(\int_0^t \gamma e^{-\gamma (t-s)} \, \rd s\right) \\
			&\quad+\int_0^t  \Pi (a,a-t+\sigma)f(\sigma,a-t+\sigma)\,  \left(\int_\sigma^t \gamma   e^{-\gamma (t-s)} \rd s\right)   \rd \sigma \\
			&\quad + \int_0^t  \Pi_\gamma (a,a-t+s) f(s,a-t+s) \, \rd s\,.
		\end{align*}
		Since
		\begin{align}\label{99s}
			\int_0^\sigma \gamma   e^{-\gamma s} \rd s=1- e^{-\gamma \sigma}\,,\quad \sigma\ge 0\,,
		\end{align}
		we deduce that indeed
		\begin{align}\label{99r}
			W_{z ,\gamma W+f}^{\gamma,h}(t,a)&=
			\Pi (a,a-t) z  (a-t) + \int_0^t  \Pi(a,a-t+s) f(s,a-t+s) \, \rd s =W_{z ,f}^{0,h}(t,a)
		\end{align}
		for $a\ge t$.\\
		
		{\bf (ii)} Next, consider the case $t > a$ so that \eqref{4998}  yields
		\begin{align}\label{99t}
			W_{z ,\gamma W+f}^{\gamma,h}(t,a)=\, &\Pi_\gamma(a,0) B_{z ,\gamma W+f}^{\gamma,h}(t-a) +G_{ \gamma W  + f}^\gamma(t,a)\,.
		\end{align}
		For the second term on the right-hand side of \eqref{99t}, given in~\eqref{4999}, we compute, using again and~\eqref{4998} for $W=W_{z ,f}^{0,h}$ and~\eqref{evol}, 
		\begin{align*}
			G_{ \gamma W  + f}^\gamma(t,a)&= \int_{t-a}^t \Pi_\gamma (a,a-t+s)  \big(  \gamma W (s,a-t+s) + f(s,a-t+s) \big) \, \rd s \nonumber\\
			&= \gamma \int_{t-a}^t \Pi_\gamma (a,a-t+s)  \Pi (a-t+s,0) B_{z ,f}^{0,h} (t-a) \, \rd s \nonumber\\
			&\quad +\gamma \int_{t-a}^t \Pi_\gamma (a,a-t+s)  \int_{t-a}^s \Pi(a-t+s,a-t+\sigma) f(\sigma,a-t+\sigma) \, \rd \sigma \, \rd s \nonumber\\
			&\quad + \int_{t-a}^t \Pi_\gamma (a,a-t+s)  f(s,a-t+s) \, \rd s\nonumber
		\end{align*} 
			\begin{align*}
				\hphantom{G_{ \gamma W  + f}^\gamma(t,a)}& = \Pi (a,0) B_{z ,f}^{0,h} (t-a) \left(\int_{t-a}^t \gamma e^{-\gamma(t-s)} \, \rd s\right) \nonumber\\
			&\quad + \int_{t-a}^t \Pi (a,a-t+\sigma) f(\sigma,a-t+\sigma) \left(\int_\sigma^t \gamma e^{-\gamma (t-s)} \, \rd s \right)\, \rd \sigma \nonumber\\
			&\quad + \int_{t-a}^t \Pi_\gamma (a,a-t+s)  f(s,a-t+s) \, \rd s  \nonumber\\
			&= -\Pi_\gamma(a,0) B_{z ,f}^{0,h} (t-a)+ \Pi(a,0) B_{z ,f}^{0,h} (t-a)\\
			&\quad +\int_{t-a}^{t}\Pi(a,a-t+s)\, f(s,a-t+s)\,\rd s
		\end{align*}
		for $t > a$, where we again applied \eqref{99s} for the last equality. Hence, from \eqref{4998},
		\begin{align}
			G_{ \gamma W  + f}^\gamma(t,a)=
			-\Pi_\gamma(a,0) B_{z ,f}^{0,h} (t-a) +W(t,a)\,,\quad t > a\, .\label{103} 
		\end{align}
Recalling~\eqref{99t} we obtain
		\begin{align}\label{101}
			W_{z ,\gamma W+f}^{\gamma,h}(t,a)&= \Pi_\gamma(a,0) \Big(B_{z ,\gamma W+f}^{\gamma,h}(t-a)-B_{z ,f}^{0,h} (t-a)\Big) +W(t,a)\,,\quad t > a\,,
		\end{align}
		where $W=W_{z ,f}^{0,h}$. We now claim that $	B_{z ,f}^{0,h}=B_{z ,\gamma W+f}^{\gamma,h}$. To this end, we use  identity~\eqref{key} for $B_{z ,f}^{0,h}(t)$ and replace therein $W=W_{z ,f}^{0,h}$ by formula~\eqref{103} for $t > a$ respectively by the formula
		\begin{align*}
			W(t,a)=W_{z ,\gamma W+f}^{\gamma,h}(t,a)=
			\Pi_\gamma (a,a-t) z  (a-t) + G_{ \gamma W  + f}^\gamma(t,a)
		\end{align*}
		for $t < a$ stemming from~\eqref{99r} and~\eqref{4998}. This yields
	\begin{equation*}
			\begin{split}
				B_{z ,f}^{0,h}(t)\,  =\, & \mathcal{M}_\phi\big(W(t,\cdot)\big) +h(t)\\
				=\, & \int_0^t  \mathfrak{b}_\phi(a)\, W(t,a)\, \rd
				a
				+\int_t^{a_m}  \mathfrak{b}_\phi(a)\, W(t,a)\, \rd a	+ h(t)\\
				=\, & \int_0^t  \mathfrak{b}_\phi(a)\, \Pi_\gamma(a,0)B_{z ,f}^{0,h}(t-a)\, \rd
				a  +\int_t^{a_m}  \mathfrak{b}_\phi(a)\, \Pi_\gamma(a,a-t) z (a-t)\, \rd a\\
				&+\mathcal{M}_\phi\big(G_{ \gamma W  + f}^\gamma(t,\cdot)\big)	+ h(t)
			\end{split}
		\end{equation*}
		for $t\in [0,T]$. That is, $B_{z ,f}^{0,h}$ satisfies the same equation as $B_{z ,\gamma W+f}^{\gamma,h}$. Consequently, we indeed have
			$B_{z ,f}^{0,h}=B_{z ,\gamma W+f}^{\gamma,h}$
		by \Cref{L1}. Therefore, \eqref{101} implies
		\begin{align*}
			W_{z ,\gamma W+f}^{\gamma,h}(t,a)=W(t,a)\,,\quad t > a\,.
		\end{align*}
		This yields the assertion.
\end{proof}

%%%%%%%%%%%%%%%%%%%%%%%%%%%%%%%%%%%%%%%%%%%%%%%
%%%%%%%%%%%%%%%%%%%%%%%%%%%%%%%%%%%%%%%%%%%%%%%

\section{Stability Estimates: Proof of \Cref{T2}}\label{Sec:6}

%%%%%%%%%%%%%%%%%%%%%%%%%%%%%%%%%%%%%%%%%%%%%%%
%%%%%%%%%%%%%%%%%%%%%%%%%%%%%%%%%%%%%%%%%%%%%%%

We shall now continue our investigation from \Cref{Sec:DeriveLin} and complete the proof of \Cref{T2}. Based on the findings from the previous section we first derive another representation of \mbox{$w=u(\cdot;u_0)-\phi$} from \Cref{intermediate} which is key for the stability estimates.

\begin{prop}\label{representation}
	Let $w=u(\cdot;u_0)-\phi$ and $w_0=u_0-\phi$ for $u_0\in \E_\alpha$ fixed.
	Then, using the notation of \Cref{intermediate}, the function $w\in C(I(u_0),\E_\alpha)$ can be written as
\begin{equation}
	\begin{split}\label{20}
	w(t)&=\T_\phi (t)w_0 +\int_0^t \T_\phi (t-s) \left( \big(\gamma+\partial F(\phi)\big)W_{0,0}^{\gamma,h_w}(s,\cdot) +R_F(w(s))\right)\,\rd s + W_{0,0}^{\gamma,h_w}(t,\cdot)
\end{split}
\end{equation}
for $t\in I(u_0)$ and every $\gamma\in\R$, where the strongly continuous semigroup $(\T_\phi (t))_{t\ge 0}$ is introduced in \Cref{LemmaH} and $W_{0,0}^{\gamma,h_w}$ in~\eqref{4998}.
\end{prop}

\begin{proof} It readily follows from \Cref{intermediate} and~\eqref{1000} that $w$ can be expressed as
$$
w(t)=W_{w_0,\partial F(\phi)w+R_F(w)}^{0,h_w}(t,\cdot)\,,\quad t\in I(u_0)\,.
$$
That is, using \Cref{LemmaB},
\begin{align}\label{18}
w(t)=W_{w_0,\partial F(\phi)w+R_F(w)}^{0,h_w}(t,\cdot)=W_{w_0,\gamma w+\partial F(\phi)w+R_F(w)}^{\gamma,h_w}(t,\cdot)\,,\quad t\in I(u_0)\,,
\end{align}
for $\gamma\in\R$ arbitrarily fixed. Consequently, \eqref{18} and the representation formula of \Cref{LemmaB19} yield
\begin{equation*}
	w(t)=e^{-\gamma t}\mS (t)w_0 +\int_0^t e^{-\gamma (t-s)}\mS (t-s)\, \big((\gamma+\partial F(\phi))w+R_F(w)\big)(s)\,\rd s +W_{0,0}^{\gamma,h_w}(t,\cdot) 
\end{equation*}
for $t\in I(u_0)$. Equivalently, we can write
\begin{align*}
w(t)-W_{0,0}^{\gamma,h_w}(t,\cdot)&=e^{-\gamma t}\mS (t)w_0\\
&\quad  +\int_0^t e^{-\gamma (t-s)}\mS (t-s) \left((\gamma+\partial F(\phi))\big( w(s)-W_{0,0}^{\gamma,h_w}(s,\cdot)\big)\right.\\
& \qquad\qquad\qquad\qquad\qquad\qquad \left. + (\gamma+\partial F(\phi))W_{0,0}^{\gamma,h_w}(s,\cdot) +R_F(w(s))\right)\,\rd s 
\end{align*}
for $t\in I(u_0)$ and then invoke \Cref{LemmaH} (note that $W_{0,0}^{\gamma,h_w}(0,\cdot)=0$) to deduce that
\begin{align*}
	w(t)-W_{0,0}^{\gamma,h_w}(t,\cdot)&=\T_\phi (t)w_0 +\int_0^t \T_\phi (t-s) \left( (\gamma+\partial F(\phi))W_{0,0}^{\gamma,h_w}(s,\cdot) +R_F(w(s))\right)\,\rd s 
\end{align*}
for $t\in I(u_0)$ as claimed.
\end{proof}

Recall from \Cref{LemmaH} that there are $N_\alpha=N_\alpha(\phi)\ge 1$ and $\omega_\alpha:=\omega_\alpha(\phi)\in\R$ such that
	\begin{equation}\label{E3oooo}
	\|\T_\phi(t)\|_{\ml(\E_\alpha)}+ t^{\alpha}\,\|\T_\phi(t)\|_{\ml(\E_0,\E_\alpha)}\le N_\alpha\, e^{ -\omega_\alpha t}\,,\quad t> 0\,.
\end{equation}
The crucial assumption now is that
\begin{equation}\label{E3ooooh}
\omega_\alpha=\omega_\alpha(\phi)>0
\end{equation}
ensuring an exponential decay of the  semigroup $(\T_\phi (t))_{t\ge 0}$ associated with the linearization of problem~\eqref{P}. \\

As a last preparation let us prove the following simple result.

\begin{lem}\label{LemmaJ}
	Let $\beta>0$ and 
	$$
	p_{\alpha,\beta}(r):=\int_0^r e^{-\beta s}\, (r-s)^{-\alpha}\, s^{-\alpha}\, \rd s\,,\quad r>0\,.
	$$
	There is $c_{\alpha,\beta}>0$ such that
	\begin{equation*}%\label{27}
	p_{\alpha,\beta}(r)\le 	c_{\alpha,\beta} r^{-\alpha}\,,\quad r>0\,.
	\end{equation*}
\end{lem}

\begin{proof}
Noticing that
\begin{align*}
	p_{\alpha,\beta}(r)&\le 2^\alpha \, r^{-\alpha} \int_0^{r/2} e^{-\beta s}\,  s^{-\alpha}\, \rd s 	
	+2^\alpha \,r^{-\alpha} \int_{r/2}^r e^{-\beta s}\, (r-s)^{-\alpha}\,  \rd s \\
	&\le 2^\alpha \beta^{\alpha-1}\Gamma(1-\alpha) \, r^{-\alpha} 
	+ \frac{2^{2\alpha-1}}{1-\alpha} r^{1-\alpha} e^{-\beta r/2}\, r^{-\alpha}
\end{align*}
for $r>0$, the assertion follows.
	
%%%%%%%%%%%%%%%%%%%%
%On the one hand, for $0<r\le 2$,
%$$
%p_{\alpha,\beta}(r)\le \int_0^r (r-s)^{-\alpha}\, s^{-\alpha}\, \rd s= r^{1-2\alpha} \, %\mathsf{B}(1-\alpha,1-\alpha)
%$$
%so that
%$$
%r^{\alpha}\, p_{\alpha,\beta}(r)\le  2^{1-\alpha} \, \mathsf{B}(1-\alpha,1-\alpha)\,,\quad 0<r\le 2\,.
%$$
%On the other hand, for $r>2$,
%\begin{align*}
%p_{\alpha,\beta}(r)&=\left(\int_0^1+\int_1^{r/2}+\int_{r/2}^r\right) e^{-\beta s}\, (r-s)^{-\alpha}\, %s^{-\alpha}\, \rd s \\
%&\le (r-1)^{-\alpha} \int_0^1  s^{-\alpha}\, \rd s +\left(\frac{r}{2}\right)^{-\alpha} \int_1^{r/2} e^{-\beta s}\, v\rd s + \left(\frac{r}{2}\right)^{-\alpha}\int_{r/2}^r e^{-\beta s}\, (r-s)^{-\alpha}\,\rd s
%\end{align*}
%and hence
%$$
%r^{\alpha}\, p_{\alpha,\beta}(r)\le  \frac{1}{1-\alpha} \left(\frac{r}{r-1}\right)^{\alpha} %+\frac{2^\alpha}{\beta}+ \frac{2^{2\alpha-1}}{1-\alpha}r^{1-\alpha} e^{-\beta r/2}\le c_{\alpha,\beta}\,,\quad %r>2\,.
%$$
%This entails the claim.
\end{proof}

We are now in a position to finish off the proof of \Cref{T2}.

%%%%%%%%%%%%%%%%%%%%%%%%%%%%%%%%%%%

\subsection*{Proof of \Cref{T2}}

%%%%%%%%%%%%%%%%%%%%%%%%%%%%%%%%%%%

Recall that we impose~\eqref{E3ooooh}.
According to~\eqref{14E2} and~\eqref{14G2} there are two increasing functions
$ d_b,d_F\in C(\R^+,\R^+)$ with $d_b(0)=d_F(0)=0$ and
\begin{align}\label{21a}
	\|R_b(v)\|_{\E_0} &\le d_b(r)\,\| v\|_{\E_\alpha}\,,\quad \| v\|_{\E_\alpha}\le r\,,
\end{align}
and 
\begin{align}\label{21b}
	\|R_F(v)\|_{\E_0} &\le d_F(r)\,\| v\|_{\E_\alpha}\,,\quad \| v\|_{\E_\alpha}\le r\,.
\end{align}
%Set 
%$$
%\nu_\alpha:=\max\{1, \|\nu\|_{L_{\infty}(J,\ml(E_\alpha))}\}\,.
%$$
Let $r>0$ be fixed (chosen small enough later; see~\eqref{30}) and consider now  $u_0\in \E_\alpha$ and $w_0=u_0-\phi$ such that
$\| w_0\|_{\E_\alpha}\le r/2$. Since $w\in C(I(u_0),\E_\alpha)$,
$$
t_1:=\sup\big\{t\in I(u_0)\,;\, \| w(s)\|_{\E_\alpha}\le r\ \text{for}\ 0\le s\le t\big\} >0\,.
$$ 
%hence
%\begin{align}\label{22}
%		\| \bar w(t)\|_{E_\alpha}\le \nu_\alpha\,\| w(t)\|_{\E_\alpha}\le r\,,\quad t\in [0,t_1]\,.
%\end{align}
We infer from~\eqref{21a} and the definition of $h_w$ in \Cref{intermediate} 
\begin{align}\label{23}
	\| h_w(t)\|_{E_0}\le \| R_b(w(t))\|_{\E_0}\le d_b(r)\,  \| w(t)\|_{\E_\alpha}\, \,,\quad t\in [0,t_1]\,.
\end{align}
Let $t\in [0,t_1)$ be fixed in the following. Denoting by $\mu>0$ the constant from \Cref{LemmaF}, we choose now $\gamma\in\R$ such that
\begin{align}\label{2i3}
-\beta:=\omega_\alpha+\varpi+\mu-\gamma <0\,.
\end{align} 
We then use the representation formula~\eqref{20} for this $\gamma$ along with~\eqref{E3oooo} and~\eqref{21b} to derive
	\begin{align*}
	\|w(t)\|_{\E_\alpha} \le\, & \|\T_\phi (t)\|_{\ml(\E_\alpha)} \,\|w_0\|_{\E_\alpha} + \|W_{0,0}^{\gamma,h_w}(t,\cdot)\|_{\E_\alpha} \\
	&  +\int_0^t \| \T_\phi (t-s)\|_{\ml(\E_0,\E_\alpha)}  \, \| \gamma+\partial F(\phi)\|_{\ml(\E_\alpha,\E_0)} \,\| W_{0,0}^{\gamma,h_w}(s,\cdot)\|_{\E_\alpha} \,\rd s \\
		&  +\int_0^t \| \T_\phi (t-s)\|_{\ml(\E_0,\E_\alpha)}  \, \| R_F(w(s)) \|_{\E_0} \,\rd s \\
	\le\, & N_\alpha\,e^{-\omega_\alpha t} \,\|w_0\|_{\E_\alpha} + \|W_{0,0}^{\gamma,h_w}(t,\cdot)\|_{\E_\alpha} \\
	&  + N_\alpha\, \| \gamma+\partial F(\phi)\|_{\ml(\E_\alpha,\E_0)} \int_0^t   (t-s)^{-\alpha}  \, e^{-\omega_\alpha (t-s)}\, \|W_{0,0}^{\gamma,h_w}(s,\cdot)\|_{\E_\alpha} \,\rd s\\
	&  + N_\alpha\, d_F(r) \int_0^t   (t-s)^{-\alpha}  \, e^{-\omega_\alpha (t-s)}\, \|w(s)\|_{\E_\alpha} \,\rd s\,. 
\end{align*}
Set
$$
m_\phi:=\| \gamma+\partial F(\phi)\|_{\ml(\E_\alpha,\E_0)}\,.
$$
We then invoke \Cref{LemmaF} (with $\mu>0$ and $c_0>0$ as therein) and~\eqref{23} to obtain
\begin{align}
	\|w(t)\|_{\E_\alpha}  \le\, & N_\alpha\,e^{-\omega_\alpha t} \,\|w_0\|_{\E_\alpha} + c_0\int_0^t e^{(\varpi+\mu-\gamma)(t-a)}\,(t-a)^{-\alpha}\,\|h_w(a)\|_{E_0}\,\rd a \nonumber\\
	&  + N_\alpha\, m_\phi \int_0^t (t-s)^{-\alpha}\, e^{-\omega_\alpha (t-s)} \int_0^s   (s-a)^{-\alpha}  \, e^{(\varpi+\mu-\gamma) (s-a)}\, \| h_w(a)\|_{E_0} \rd a \,\rd s\nonumber\\
	&  + N_\alpha\, d_F(r) \int_0^t   (t-s)^{-\alpha}  \, e^{-\omega_\alpha (t-s)}\, \|w(s)\|_{\E_\alpha} \,\rd s \nonumber\\
	\le\, & N_\alpha\,e^{-\omega_\alpha t} \,\|w_0\|_{\E_\alpha} + c_0\, d_b(r)\,\int_0^t e^{(\varpi+\mu-\gamma)(t-a)}\,(t-a)^{-\alpha}\,  \| w(a)\|_{\E_\alpha}\,\rd a \nonumber\\
	&  + N_\alpha\, m_\phi\, d_b(r)\, \int_0^t (t-s)^{-\alpha}\, e^{-\omega_\alpha (t-s)} \int_0^s   (s-a)^{-\alpha}  \, e^{(\varpi+\mu-\gamma) (s-a)}\, \| w(a)\|_{\E_\alpha} \rd a \,\rd s \nonumber\\
	&  + N_\alpha\, d_F(r) \int_0^t   (t-s)^{-\alpha}  \, e^{-\omega_\alpha (t-s)}\, \|w(s)\|_{\E_\alpha} \,\rd  s\,. \label{611}	
\end{align}
As for the third term in~\eqref{611} we note that
\begin{align*}
\int_0^t (t-s)^{-\alpha}\, & e^{-\omega_\alpha (t-s)} \int_0^s   (s-a)^{-\alpha}  \, e^{(\varpi+\mu-\gamma) (s-a)}\, \| w(a)\|_{\E_\alpha} \rd a \,\rd s\\
& =
e^{-\omega_\alpha t}\int_0^t e^{\omega_\alpha a}\, \| w(a)\|_{\E_\alpha} \int_a^t e^{(\omega_\alpha+\varpi+\mu-\gamma) (s-a)}\,(t-s)^{-\alpha}\,(s-a)^{-\alpha}\, \rd s\,  \rd a\\
&\le c_{\alpha,\beta}\, e^{-\omega_\alpha t}\int_0^t (t-a)^{-\alpha}\, e^{\omega_\alpha a}\, \| w(a)\|_{\E_\alpha}\,  \rd a\,,
\end{align*}
where we used~\eqref{2i3} and \Cref{LemmaJ} for the last estimate. Using this in~\eqref{611} and again~\eqref{2i3} in the second term of~\eqref{611} to drop part of the exponential we infer that
\begin{align*}
	e^{\omega_\alpha t}\|w(t)\|_{\E_\alpha} 	\le\, & N_\alpha\, \|w_0\|_{\E_\alpha} \\
	&  + \big( c_0\,  d_b(r)+N_\alpha\, m_\phi\, d_b(r)\, c_{\alpha,\beta} + N_\alpha\, d_F(r)\big) \int_0^t (t-a)^{-\alpha}\, e^{\omega_\alpha a}\, \| w(a)\|_{\E_\alpha}\,  \rd a  
\end{align*}
for $t\in [0,t_1)$.
Fix now $\omega\in (0,\omega_\alpha)$ and choose $r>0$ such that
\begin{align}\label{30}
c_0\, d_b(r)+N_\alpha\, m_\phi\, d_b(r)\, c_{\alpha,\beta} + N_\alpha\, d_F(r)\le \left(\frac{2(\omega_\alpha-\omega)}{3}\right)^{1-\alpha} \frac{1}{\Gamma(1-\alpha)}=:\sigma_\alpha\,.
\end{align}
Then
\begin{align*}
	e^{\omega_\alpha t}\|w(t)\|_{\E_\alpha} 	\le  N_\alpha\, \|w_0\|_{\E_\alpha}  + \sigma_\alpha \int_0^t (t-a)^{-\alpha}\, e^{\omega_\alpha a}\, \| w(a)\|_{\E_\alpha}\,  \rd a  \,,\quad t \in [0,t_1)\,,
\end{align*}
so that Gronwall's inequality \cite[II.Theorem 3.3.1]{LQPP} implies that there is a constant $k_0\ge 1$ such that
\begin{align*}%\label{31}
	e^{\omega_\alpha t}\|w(t)\|_{\E_\alpha} 	\le  k_0\,N_\alpha\, \|w_0\|_{\E_\alpha}  e^{\frac{3}{2}(\Gamma(1-\alpha)\sigma_\alpha )^{1/(1-\alpha)}t}\,,\quad t\in [0,t_1)\,. 
\end{align*}
That is,  by choice of $\sigma_\alpha$,
\begin{align}\label{31}
	\|w(t)\|_{\E_\alpha} 	\le\,  k_0\,N_\alpha\, \|w_0\|_{\E_\alpha}  e^{-\omega t}\,,\quad t\in [0,t_1)\,. 
\end{align}
Consequently, for every $w_0=u_0-\phi\in \E_\alpha$ with
$$
  \|w_0\|_{\E_\alpha} \le \frac{r}{2k_0\,N_\alpha}\le \frac{r}{2}
$$
we have
\begin{align*}
	\|w(t)\|_{\E_\alpha} 	\le  \frac{r}{2} \, e^{-\omega t}\le \frac{r}{2}\,,\quad t\in [0,t_1)\,. 
\end{align*}
By definition of $t_1$, we conclude $t_1=\sup I(u_0)$ and then $I(u_0)=\R^+$ according to \Cref{T1}. Therefore, invoking~\eqref{31} we have proven that
\begin{align*}
	\|u(t;u^0)-\phi\|_{\E_\alpha} 	\le k_0\,N_\alpha\, \|u_0-\phi\|_{\E_\alpha}  e^{-\omega t}\,,\quad t\in \R^+\,,
\end{align*}
whenever
$$
\|u_0-\phi\|_{\E_\alpha} \le \frac{r}{2 k_0 N_\alpha}\,.
$$
 This completes the proof of \Cref{T2}.

%%%%%%%%%%%%%%%%%%%%%%%%%%%%%%%%%%%%%%%%%%%%%%%%%%%%%%%%%%%%%%%
%%%%%%%%%%%%%%%%%%%%%%%%%%%%%%%%%%%%%%%%%%%%%%%%%%%%%%%%%%%%%%%

\section{Examples}\label{Sec:Exmp}

%%%%%%%%%%%%%%%%%%%%%%%%%%%%%%%%%%%%%%%%%%%%%%%%%%%%%%%%%%%%%%%
%%%%%%%%%%%%%%%%%%%%%%%%%%%%%%%%%%%%%%%%%%%%%%%%%%%%%%%%%%%%%%%

We shed some light on the assumptions required for \Cref{T2} and consider particular cases. For simplicity we assume throughout that $a_m<\infty$.

%%%%%%%%%%%%%%%%%%%%%%%%%%%%%%%%%%%%%%%%%%%%%%%%%%%%%%%%%%%%%%%
\subsection*{Stability of the Trivial Equilibrium}
%%%%%%%%%%%%%%%%%%%%%%%%%%%%%%%%%%%%%%%%%%%%%%%%%%%%%%%%%%%%%%%

Consider the trivial equilibrium $\phi=0$.  Note that then $\mathfrak{b}_\phi=b(0,\cdot)$ in \eqref{bb}. Assume (also for simplicity) that 
\begin{subequations}\label{MMM}
\begin{equation}\label{71cc}
m(0,\cdot)\in C^\rho(J,\mathcal{L}_+(E_\alpha,E_0))\,.
\end{equation}
Then  
$$
A_0:=A-m(0,\cdot)\in C^\rho(J,\mathcal{H}(E_1,E_0))
$$ 
generates an evolution operator $\Pi_0$ on $E_0$ due to \cite[II.Corollary~4.4.2]{LQPP}. Suppose further (see~\eqref{B3}) that
\begin{equation}\label{71cx}
	b(0,\cdot)\in   L_{\infty}\big(J,\ml(E_\theta)\big)\,,\quad \theta\in [0,1]\,,
\end{equation}
and that
\begin{equation}\label{71cy}
	b(0,a){\Pi_0}(a,0)\in\ml(E_0)  \text{ is strongly positive for $a$ in a subset of $J$ of positive measure}\,.
\end{equation}
\end{subequations}
Then  \cite[Corollary 5.3]{WalkerIUMJ} implies that the  growth bound of the semigroup $(\T_0(t))_{t\ge 0}$ from Proposition~\ref{LemmaH} coincides with the spectral bound of its generator and also with the unique $\lambda_0\in \R$ such that
\begin{equation}\label{r0}
r(Q_{\lambda_0})=1\,,
\end{equation}
where $r( Q_{\lambda})$ denotes the spectral radius of the strongly positive compact operator
\begin{equation}\label{Q0}
Q_{\lambda}:=\int_0^{a_m} b(0,a)\,\Pi_{0,\lambda}(a,0)\, \rd a\in\ml(E_0)\,, \quad \lambda\in \R\,.
\end{equation}
Roughly speaking $r( Q_{\lambda})$ may be interpreted as the expected number of offspring per individual during its life span.
Actually, the mapping $\lambda\mapsto r(Q_{\lambda})$ is continuous and strictly decreasing \cite[Lemma 4.1]{WalkerIUMJ}. Thus, the assumption $\lambda_0<0$ in~\eqref{r0} is equivalent to $r(Q_0)<1$, and hence to $\omega_\alpha(0)>0$ in~\eqref{E3ooo} (see \Cref{R14}).

Consequently, we can state the asymptotic stability of the trivial equilibrium as follows:
	
	\begin{cor}\label{T3}
		Let $\alpha\in [0,1)$ and suppose \eqref{A1a}, \eqref{A1b}, \eqref{A1c}, and \eqref{A1d}. Moreover, suppose \eqref{B1} and \eqref{B2} are satisfied for $\phi=0$ and assume \eqref{MMM}.
		If $r(Q_0)<1$, then the trivial equilibrium $\phi=0$ is asymptotically exponentially stable in~$\E_\alpha$.
\end{cor}

We shall get back to this below. It is also worth noting that the semigroup $(\T_0(t))_{t\ge 0}$ has ansychronous exponential growth if conversely $r(Q_0)>1$, see \cite[Corollary 2.6]{WalkerIUMJ}.

%%%%%%%%%%%%%%%%%%%%%%%%%%%%%%%%%%%%%%%%%%%%%%%%%%%%%%%%%%%%%%%
\subsection*{Application to Problem~\eqref{Eu1a}}
%%%%%%%%%%%%%%%%%%%%%%%%%%%%%%%%%%%%%%%%%%%%%%%%%%%%%%%%%%%%%%%

Let $\Omega\subset\R^n$ be bounded with smooth boundary and consider problem~\eqref{Eu1a} in the form
\begin{subequations}	\label{P99}
	\begin{align}
		\partial_t u+\partial_a u&=\mathrm{div}_x\big(d(a,x)\nabla_xu\big)-m\big(\bar u(t,x),a\big)u\ , && t>0\, , &  a\in (0,a_m)\, ,& & x\in\Om\, ,\\
		u(t,0,x)&=\int_0^{a_m} b\big(\bar u(t,x),a\big)u(t,a,x)\,\rd a\, ,& & t>0\, , & & & x\in\Om\, ,\\
		\partial_N u(t,a,x)&=0\ ,& & t>0\, , &  a\in (0,a_m)\, ,& & x\in\partial\Om\, ,\\
		u(0,a,x)&=u_0(a,x)\ ,& & &  a\in (0,a_m)\, , & & x\in\Om\,,
	\end{align}
\end{subequations}
with
$$
\bar v (x)=\int_0^{a_m} \nu(a,x)\,v(a,x)\,\rd a\,,\quad x\in\Omega\,.
$$
We assume for the data (striving rather for simple than optimal conditions) that
\begin{subequations}\label{j}
\begin{align}
&d\in C^{\rho,1}(J\times\bar\Omega,(0,\infty))\,,\label{j1}\\
&b,m\in C^{4,0}(\R\times J,\R^+)\,,\label{j2}\\
&\nu\in C^{0,2}(J\times\bar\Omega,\R^+)\,.\label{j3}
\end{align}
\end{subequations}
For instance, $\nu\equiv 1$ is a possible choice.
Let $q>n$ and set $E_0:=L_q:=L_q(\Omega)$ and 
$$
E_1:=W_{q,N}^2:=\{v\in W_{q}^2(\Omega)\,;\; \partial_N w=0 \text{ on } \partial\Omega\}\,.
$$ 
Then $E_1$ is compactly embedded in $E_0$ and, for real interpolation,
\begin{equation}\label{interpol}
E_\theta:=\big(L_q,W_{q,N}^2)_{\theta,q} \doteq W_{q,N}^{2\theta}:=\left\{\begin{array}{ll} \{v\in W_{q}^{2\theta}(\Omega)\,;\; \partial_N w=0 \text{ on } \partial\Omega\}\,, & 1+1/q<2\theta\le 2\,,\\[3pt]
	 W_{q}^{2\theta}(\Omega)\,, & 0\le 2\theta<1+1/q\,.\end{array} \right.
\end{equation}
Setting
$$
A(a)w:=\mathrm{div}_x\big(d(a,\cdot)\nabla_xw\big)\, ,\quad w\in W_{q,N}^2\,,\quad a\in J=[0,a_m]\,,
$$
it follows from~\eqref{j1} that $A \in  C^\rho\big(J,\mathcal{H}(W_{q,N}^2,L_q)\big)$ so that~\eqref{A1a} is valid. Moreover, the maximum principle ensures~\eqref{A1aa} while \cite[II.Lemma 5.1.3]{LQPP} entails~\eqref{EO}. Fixing $2\alpha\in (n/q,2)\setminus\{1+1/q\}$, it follows from~\eqref{j2} and \cite[Proposition 4.1]{WalkerAMPA} that
$$
[v\mapsto b(v,\cdot)]\,,\, [v\mapsto m(v,\cdot)] \in C^1\big(W_{q,N}^{2\alpha},L_\infty(J,W_{q,N}^{2\eta})\big)\,,\quad 0\le 2\eta<2\alpha\,,\quad 2\eta\not= 1+1/q\,,
$$
with
\begin{align}\label{diff}
\big(\partial b(v,\cdot)[h]\big)(a)(x)=\partial_1 b(v(x),a)h(x)\,,\qquad (a,x)\in J\times \Om\,, \quad  v, h\in W_{q,N}^{2\alpha}\,.
\end{align}
In particular, using that pointwise multiplication is obviously continuous as a mapping $$W_{q,N}^{2\eta}\times W_{q,N}^{2\alpha}\to L_q$$ we infer from \eqref{interpol} that~\eqref{B1} and~\eqref{B2} are valid and hence also ~\eqref{A1b} and~\eqref{A1c}. Moreover, if $\phi\in \E_1=L_1(J,W_{q,N}^2)$ is an arbitrary equilibrium, then 
$$
\bar\phi=\int_0^{a_m} \nu(a,\cdot)\,\phi(a)\,\rd a\in W_{q,N}^2
$$
owing to~\eqref{j3}, hence $b(\bar\phi,\cdot) \in L_\infty(J, W_{q,N}^2)$. Since pointwise multiplication 
$$
W_{q,N}^{2}\times W_{q,N}^{2\alpha}\to W_{q,N}^{2\alpha}
$$ 
is continuous \cite{AmannMult} we deduce~\eqref{B3}. Moreover, since $\partial_1 b(\bar\phi,\cdot)\in L_\infty(J,W_{q,N}^{2-\ve})$ for every $\ve>0$ small and since pointwise multiplication $W_{q,N}^{2-\ve}\times W_{q,N}^{2\theta}\to W_{q,N}^{2\theta}$ is continuous for $\theta=0,\alpha$, we also deduce~\eqref{B2b}. Clearly, \eqref{j3} implies~\eqref{A1d}. Also note that if 
\begin{equation}\label{pl}
b(z,a)>0\,,\quad (z,a)\in\R\times J\,,
\end{equation}
 then \cite[Section 13]{DanersKochMedina}  implies~\eqref{71cy} while ~\eqref{71cx} follows from the above observations.

In particular, assumptions~\eqref{A} and~\eqref{B} are all satisfied.

%%%%%%%%%%%%%%%%%%%%%%%%%%%%%%%%%%%%%%%%%%%%%%%%%%%%%%%%%%%%%%%
\subsubsection*{Stability of the Trivial Equilibrium Revisited}
%%%%%%%%%%%%%%%%%%%%%%%%%%%%%%%%%%%%%%%%%%%%%%%%%%%%%%%%%%%%%%%

The verification of the crucial assumption~\eqref{E3ooooh} is not straightforward (and depends, of course, on the concrete equilibrium). However, for the trivial equilibrium $\phi=0$ this is possible under suitable assumptions as seen previously. We use the same notation as above and still suppose~\eqref{j} and \eqref{pl} so that assumptions~\eqref{A} and~\eqref{B} as well as~~\eqref{71cy} and ~\eqref{71cx} are satisfied by the previous considerations (\eqref{71cc} is not required here since $m(0,\cdot)$ is independent of $x\in\Omega$). If $\Pi$ still denotes the evolution operator on $L_q$ associated with
$$
A(a)w=\mathrm{div}_x\big(d(a,\cdot)\nabla_xw\big)\, ,\quad w\in W_{q,N}^2\,,\quad a\in J\,,
$$
then the evolution operator $\Pi_0$ associated with $A_0=A-m(0,\cdot)$ is simply
$$
\Pi_0(a,\sigma)=e^{-\int_\sigma^a m(0,s)\rd s}\,\Pi(a,\sigma)\,,\quad 0\le \sigma\le a\in J\,,
$$
so that the operator $Q_0$ from~\eqref{Q0} is
$$
Q_0=\int_0^{a_m} b(0,a)\, e^{-\int_0^a m(0,s)\rd s}\,\Pi(a,0)\,\rd a\,.
$$
Since  $\Pi(a,0)\mathbf{1}=\mathbf{1}$ for $\mathbf{1}:=[x\mapsto 1]\in W_{q,N}^2$, it follows that
$$
Q_0\mathbf{1}=  \int_0^{a_m} b(0,a)\, e^{-\int_0^a m(0,s)\rd s}\,\rd a\,\mathbf{1}\,.
$$
That is, $\mathbf{1}$ is a positive eigenfunction of the strongly positive compact operator $Q_0$  so that the Krein-Rutman theorem (e.g., see \cite[Theorem 12.3]{DanersKochMedina}) entails that
$$
r(Q_0)= \int_0^{a_m} b(0,a)\, e^{-\int_0^a m(0,s)\rd s}\,\rd a\,.
$$
Consequently, we obtain from \Cref{T3}:

\begin{cor}\label{CC}
Assume \eqref{j}, and \eqref{pl}. If 
$$
\int_0^{a_m} b(0,a)\, e^{-\int_0^a m(0,s)\rd s}\,\rd a<1\,,
$$
then $\phi=0$ is an  asymptotically exponentially stable equilibrium of problem~\eqref{P99} in the phase space $L_1\big((0,a_m),W_{q,N}^{2\alpha}(\Omega)\big)$ for $2\alpha\in (n/q,2)\setminus\{1+1/q\}$.
\end{cor}

For instance, if the death rate dominates the birth rate in the sense that
$$
b(0,a)\le m(0,a)\,,\quad a\in J\,,
$$
then
$$
\int_0^{a_m} b(0,a)\, e^{-\int_0^a m(0,s)\rd s}\,\rd a \le \int_0^{a_m} m(0,a)\, e^{-\int_0^a m(0,s)\rd s}\,\rd a=1-e^{-\int_0^{a_m} m(0,s)\rd s} <1
$$
as required in \Cref{CC}.

%%%%%%%%%%%%%%%%%%%%%%%%%%%%%%%%%%%%%%%%%%%%%%%%%%%%%%%%%%%%%%%
\subsection*{Stability of a Nontrivial Equilibrium of Problem~\eqref{Eu1a}}
%%%%%%%%%%%%%%%%%%%%%%%%%%%%%%%%%%%%%%%%%%%%%%%%%%%%%%%%%%%%%%%

We only sketch a particular case for a nontrivial positive equilibrium of problem~\eqref{Eu1a}.
Let $\Omega\subset\R^n$ be bounded with smooth boundary and consider
\begin{subequations}\label{PP}
		\begin{align}
	\partial_t u+\partial_a u&=\mathrm{div}_x\big(d(a,x)\nabla_xu\big)-m\big(a,x\big)u\ , && t>0\, , &  a\in (0,a_m)\, ,& & x\in\Om\, ,\\
	u(t,0,x)&=\int_0^{a_m} b\big(\bar u(t,x),a\big)u(t,a,x)\,\rd a\, ,& & t>0\, , & & & x\in\Om\, ,\\
	\partial_N u(t,a,x)&=0\ ,& & t>0\, , &  a\in (0,a_m)\, ,& & x\in\partial\Om\, ,\\
	u(0,a,x)&=u_0(a,x)\ ,& & &  a\in (0,a_m)\, , & & x\in\Om\,,
\end{align}
\end{subequations}
with
$$
\bar v (x)=\int_0^{a_m} \nu(a,x)\,v(a,x)\,\rd a\,,\quad x\in\Omega\,.
$$
For the data we assume
\begin{subequations}\label{J}
	\begin{align}
		&d\in C^{\rho,1}(J\times\bar\Omega,(0,\infty))\,,\label{J1}\\
		&b\in C^{4,0}(\R\times J,\R^+)\,,\quad b>0\,,\label{J2}\\
		& m\in C^{\rho,2}(J\times \bar\Omega,\R^+)\,,\label{J3}\\
		&\nu\in C^{0,2}(J\times\bar\Omega,\R^+)\,.\label{J4}
	\end{align}
Note that we assume $m=m(a,x)$ to be independent of $u$. Consider now a (nontrivial) positive equilibrium 
$$
\phi\in\E_1\cap C(J,E_0)=L_{1}\big(J,\ml(W_{q,N}^{2})\big)\cap C([0,a_m],L_q)
$$ 
and fix again $2\alpha\in (n/q,2)$. Then, as above,
\begin{equation*}
	b(\bar \phi,\cdot)\in   L_{\infty}\big(J,\ml(W_{q,N}^{2\theta})\big)\,,\quad 2\alpha\le 2\theta<2\,,
\end{equation*}
and 
\begin{equation*}
\big[v\mapsto \partial b(\bar \phi,\cdot)[ v]\phi\big]\in \ml\big(W_{q,N}^{2\theta}, L_1(J,W_{q,N}^{2\theta})\big)\,,\quad 2\alpha\le 2\theta<2\,,
\end{equation*}
that is, \eqref{B3} and \eqref{B2b} are satisfied for $\theta\in (\alpha,1)$. Since $m=m(a,x)$ is independent of $u$, we may put
$$
A(a)w:=\mathrm{div}_x\big(d(a,\cdot)\nabla_xw\big)-m(a,\cdot) w\, ,\quad w\in W_{q,N}^2\,,\quad a\in J\,,
$$ 
from the very beginning. Then $A \in  C^\rho\big(J,\mathcal{H}(W_{q,N}^2,L_q)\big)$ and we may interpret $F=0$ in~\eqref{2}. Then the semigroup  $(\T_\phi(t))_{t\ge 0}$ coincides with the semigroup $(\mS(t))_{t\ge 0}$ from \Cref{LemmaBB} and is thus
 eventually compact on $\E_\alpha$ according to \cite[Corollary 2.2]{WalkerIUMJ}. Hence, its growth bound coincides with the spectral bound $s_\alpha$ of its generator due to \cite[IV. Corollary 3.12]{EngelNagel}, and the spectrum of the generator consists of eigenvalues only \cite[V. Corollary 3.2]{EngelNagel}. Clearly, any such eigenvalue is also an eigenvalue of the generator of the semigroup  $(\T_\phi(t))_{t\ge 0}$ considered on~$\E_0$, hence $s_\alpha\le s_0$.  Assuming further that
\begin{align}\label{pos}
\partial_1 b(\bar\phi(x),a)\ge 0\,,\quad (a,x)\in J\times \Omega\,,
\end{align}
\end{subequations}
it follows from \cite[Theorem 2.8]{WalkerIUMJ} that $(\T_\phi(t))_{t\ge 0}$  is a positive semigroup on $\E_0$. One then argues as in \cite[Proposition 5.2]{WalkerIUMJ} to conclude that $s_0$ is an eigenvalue of the generator of $(\T_\phi(t))_{t\ge 0}$.
If  $\psi$ is a corresponding eigenfunction, it satisfies (see \cite{WalkerIUMJ})
	\begin{align}
	\partial_a\psi \, &=    (-s_0+A(a) )\psi \,,  \quad a\in J\, ,\label{wqq}\\ 
\psi(0)&=\mathcal{M}_\phi(\psi)  \label{wq}
\end{align} 
with (using~\eqref{diff})
$$
\mathcal{M}_\phi(\psi)=\int_0^{a_m} b(\bar\phi,a)\, \psi(a)\,\rd a+ \int_0^{a_m} \partial_1 b(\bar\phi,a)\, \phi(a)\,\rd a\, \int_0^{a_m} \nu(a)\, \psi(a)\,\rd a\,.
$$
Letting $\Pi_{s_0}$ denote the evolution operator associated with $-s_0+A\in  C^\rho\big(J,\mathcal{H}(W_{q,N}^2,L_q)\big)$ it follows from~\eqref{wqq} that
$$
\psi(a)=\Pi_{s_0}(a,0) \psi(0)\,,\quad a\in J\,.
$$
Plugging this into~\eqref{wq} implies that
$$
\psi(0)= Q_{\phi,{s_0}} \psi(0)\,,
$$
where
$$
Q_{\phi,{\lambda}} z:=\int_0^{a_m} b(\bar\phi,a)\, \Pi_{\lambda}(a,0)\,z\,\rd a+ \int_0^{a_m} \partial_1 b(\bar\phi,a)\, \phi(a)\,\rd a\, \int_0^{a_m} \nu(a)\, \Pi_{\lambda}(a,0)\, z\,\rd a
$$
is a compact operator on $E_0=L_q$ for $\lambda\in\R$. Hence, $1$ is an eigenvalue of $Q_{\phi,{s_0}}$ so that 
$$
1\le r(Q_{\phi,{s_0}})\,.
$$ 
By~\eqref{J2} and~\eqref{pos},
the operator $Q_{\phi,\lambda}$ is even strongly positive and the spectral radius $r(Q_{\phi,\lambda})$ is a decreasing function with respect to $\lambda$ (this is shown analogously to  \cite[Lemma 2.4, Lemma 2.5]{WalkerMOFM}). Therefore, the assumption
\begin{align}\label{ko}
r(Q_{\phi,0})<1
\end{align}
implies $s_0<0$ and hence a negative growth bound for the semigroup  $(\T_\phi(t))_{t\ge 0}$ on $\E_\alpha$. Consequently, one obtains from \Cref{T2}:

\begin{cor}
 Let $\phi\in\ L_{1}\big(J,\ml(W_{q,N}^{2})\big)\cap C([0,a_m],L_q)$ be a positive equilibrium to~\eqref{PP}	and assume~\eqref{J} and~\eqref{ko}. Then $\phi$ is asymptotically exponentially stable in $\E_\alpha=L_1(J,W_{q,N}^{2\alpha})$.
\end{cor}

%%%%%%%%%%%%%%%%
\bibliographystyle{siam}
\bibliography{AgeDiff}

\begin{thebibliography}{10}

\bibitem{AmannMult}
{\sc H.~Amann}, {\em Multiplication in {S}obolev and {B}esov spaces}, in
  Nonlinear analysis, Sc. Norm. Super. di Pisa Quaderni, Scuola Norm. Sup.,
  Pisa, 1991, pp.~27--50.

\bibitem{LQPP}
\leavevmode\vrule height 2pt depth -1.6pt width 23pt, {\em Linear and
  quasilinear parabolic problems. {V}ol. {I}}, vol.~89 of Monographs in
  Mathematics, Birkh\"{a}user Boston, Inc., Boston, MA, 1995.
\newblock Abstract linear theory.

\bibitem{DanersKochMedina}
{\sc D.~Daners and P.~Koch~Medina}, {\em Abstract evolution equations, periodic
  problems and applications}, vol.~279 of Pitman Research Notes in Mathematics
  Series, Longman Scientific \& Technical, Harlow; copublished in the United
  States with John Wiley \& Sons, Inc., New York, 1992.

\bibitem{DelgadoMolinasuarezJDE08}
{\sc M.~Delgado, M.~Molina-Becerra, and A.~Su\'{a}rez}, {\em Nonlinear
  age-dependent diffusive equations: a bifurcation approach}, J. Differential
  Equations, 244 (2008), pp.~2133--2155.

\bibitem{EngelNagel}
{\sc K.-J. Engel and R.~Nagel}, {\em One-parameter semigroups for linear
  evolution equations}, vol.~194 of Graduate Texts in Mathematics,
  Springer-Verlag, New York, 2000.
\newblock With contributions by S. Brendle, M. Campiti, T. Hahn, G. Metafune,
  G. Nickel, D. Pallara, C. Perazzoli, A. Rhandi, S. Romanelli and R.
  Schnaubelt.

\bibitem{GurtinMacCamy_MB81}
{\sc M.~E. Gurtin and R.~C. MacCamy}, {\em Diffusion models for age-structured
  populations}, Math. Biosci., 54 (1981), pp.~49--59.

\bibitem{Henry}
{\sc D.~Henry}, {\em Geometric theory of semilinear parabolic equations},
  vol.~840 of Lecture Notes in Mathematics, Springer-Verlag, Berlin-New York,
  1981.

\bibitem{KangRuanJMB21}
{\sc H.~Kang and S.~Ruan}, {\em Mathematical analysis on an age-structured
  {SIS} epidemic model with nonlocal diffusion}, J. Math. Biol., 83 (2021),
  p.~5.

\bibitem{KangRuanJDE21}
\leavevmode\vrule height 2pt depth -1.6pt width 23pt, {\em Nonlinear
  age-structured population models with nonlocal diffusion and nonlocal
  boundary conditions}, J. Differential Equations, 278 (2021), pp.~430--462.

\bibitem{KangRuan_MA21}
\leavevmode\vrule height 2pt depth -1.6pt width 23pt, {\em Principal spectral
  theory and asynchronous exponential growth for age-structured models with
  nonlocal diffusion of neumann type}, Math. Ann.,  (2021).

\bibitem{Langlais88}
{\sc M.~Langlais}, {\em Large time behavior in a nonlinear age-dependent
  population dynamics problem with spatial diffusion}, J. Math. Biol., 26
  (1988), pp.~319--346.

\bibitem{PruessJMB81}
{\sc J.~Pr\"{u}\ss}, {\em Equilibrium solutions of age-specific population
  dynamics of several species}, J. Math. Biol., 11 (1981), pp.~65--84.

\bibitem{PruessNA83}
{\sc J.~Pr\"{u}ss}, {\em Stability analysis for equilibria in age-specific
  population dynamics}, Nonlinear Anal., 7 (1983), pp.~1291--1313.

\bibitem{Rhandi}
{\sc A.~Rhandi}, {\em Positivity and stability for a population equation with
  diffusion on {$L^1$}}, Positivity, 2 (1998), pp.~101--113.

\bibitem{RhandiSchnaubelt_DCDS99}
{\sc A.~Rhandi and R.~Schnaubelt}, {\em Asymptotic behaviour of a
  non-autonomous population equation with diffusion in {$L^1$}}, Discrete
  Contin. Dynam. Systems, 5 (1999), pp.~663--683.

\bibitem{ThiemeDCDS}
{\sc H.~R. Thieme}, {\em Positive perturbation of operator semigroups: growth
  bounds, essential compactness, and asynchronous exponential growth}, Discrete
  Contin. Dynam. Systems, 4 (1998), pp.~735--764.

\bibitem{WalkerSIMA09}
{\sc {\relax Ch}.~Walker}, {\em Positive equilibrium solutions for age- and
  spatially-structured population models}, SIAM J. Math. Anal., 41 (2009),
  pp.~1366--1387.

\bibitem{WalkerDCDSA10}
\leavevmode\vrule height 2pt depth -1.6pt width 23pt, {\em Age-dependent
  equations with non-linear diffusion}, Discrete Contin. Dyn. Syst., 26 (2010),
  pp.~691--712.

\bibitem{WalkerJDE10}
\leavevmode\vrule height 2pt depth -1.6pt width 23pt, {\em Global bifurcation
  of positive equilibria in nonlinear population models}, J. Differential
  Equations, 248 (2010), pp.~1756--1776.

\bibitem{WalkerAMPA}
\leavevmode\vrule height 2pt depth -1.6pt width 23pt, {\em Bifurcation of
  positive equilibria in nonlinear structured population models with varying
  mortality rates}, Ann. Mat. Pura Appl. (4), 190 (2011), pp.~1--19.

\bibitem{WalkerCrelle}
\leavevmode\vrule height 2pt depth -1.6pt width 23pt, {\em On positive
  solutions of some system of reaction-diffusion equations with nonlocal
  initial conditions}, J. Reine Angew. Math., 660 (2011), pp.~149--179.

\bibitem{WalkerMOFM}
\leavevmode\vrule height 2pt depth -1.6pt width 23pt, {\em Some remarks on the
  asymptotic behavior of the semigroup associated with age-structured diffusive
  populations}, Monatsh. Math., 170 (2013), pp.~481--501.

\bibitem{WalkerJEPE}
\leavevmode\vrule height 2pt depth -1.6pt width 23pt, {\em Some results based
  on maximal regularity regarding population models with age and spatial
  structure}, J. Elliptic Parabol. Equ., 4 (2018), pp.~69--105.

\bibitem{WalkerIUMJ}
\leavevmode\vrule height 2pt depth -1.6pt width 23pt, {\em Properties of the
  semigroup in {$L$}$_1$~associated with age-structured diffusive populations}.
\newblock arXiv: 2109.01573, 2021.

\bibitem{WebbBook}
{\sc G.~F. Webb}, {\em Theory of nonlinear age-dependent population dynamics},
  vol.~89 of Monographs and Textbooks in Pure and Applied Mathematics, Marcel
  Dekker, Inc., New York, 1985.

\bibitem{WebbSpringer}
\leavevmode\vrule height 2pt depth -1.6pt width 23pt, {\em Population models
  structured by age, size, and spatial position}, in Structured population
  models in biology and epidemiology, vol.~1936 of Lecture Notes in Math.,
  Springer, Berlin, 2008, pp.~1--49.

\end{thebibliography}
%%%%%%%%%%%%%%%%

\end{document}